\documentclass[11pt,onecolumn]{article}
\usepackage{amscd}
\usepackage{times}
\usepackage{amsmath}
\usepackage{amssymb}
\usepackage{xspace}
\usepackage{theorem}
\usepackage{enumitem} 
\usepackage{graphicx}
\usepackage{ifpdf}
\usepackage{url,hyperref}
\usepackage{latexsym}
\usepackage{euscript}
\usepackage{xspace}
\usepackage[all]{xy}
\usepackage{color}
\usepackage{makeidx}
\usepackage{tikz}
\long\def\remove#1{}
\newtheorem{theorem}{Theorem}[section] 

\newtheorem{obs}[theorem]{Observation}

\newtheorem{definition}[theorem]{Definition}
\newtheorem{proposition}[theorem]{Proposition}
\newenvironment{proof}{{\em Proof:}}{\hfill{\hfill\rule{2mm}{2mm}}}

\newcommand {\mm}[1] {\ifmmode{#1}\else{\mbox{\(#1\)}}\fi}

\newcommand{\B}                        {\mathrm {\mathbb{B}}}

\newcommand{\img}{\mathrm img}
\newcommand{\supp}{\mathrm supp}
\newcommand{\coker} {\mathrm coker}

\newcommand{\cancel}[1]

\begin{document}

\title {Alternative to Morse-Novikov Theory  for a closed 1-form (II)}

\author{
Dan Burghelea  \thanks{
Department of Mathematics,
The Ohio State University, Columbus, OH 43210,USA.
Email: {\tt burghele@math.ohio-state.edu}}
}
\date{}

\date{}
\maketitle

 \begin{abstract}
This paper  is a continuation of \cite {BU1} and  establishes:

a) a refinement of Poincar\'e duality  
to an equality between the configurations $^{BM} \delta^\omega$ and $\delta^\omega$ resp. $^{BM} \gamma^\omega$ and $\gamma^\omega$ in complementary dimensions, 

b)  the stability property for the configurations $\delta^\omega_r,$  

c) a  result  needed for the proof of Theorems 1.2 and 1.3 in \cite{BU1} . 
\end{abstract}

\maketitle 

\setcounter{tocdepth}{1}
\tableofcontents

\section {Introduction}

\noindent  
This paper  is a continuation of \cite {BU1} and  establishes:

a) a refinement of Poincar\'e duality   
to an equality between the configurations $^{BM} \delta^\omega$ and $\delta^\omega$ resp. $^{BM} \gamma^\omega$ and $\gamma^\omega$ in complementary dimensions, 

b) the stability property for the configurations $\delta^\omega_r,$  

c)  a  result  needed for the proof of  Theorems 1.2 and 1.3 announced in \cite{BU1} . 

\noindent Since we are interested in Poincar\'e duality which has to be considered for non compact manifolds, it is necessary to involve  the functor $^{BM} H_r,$  the Borel-Moore homology with coefficients in a fixed field $\kappa,$ abbreviated BM-homology, and the natural transformation $\theta_r: H_r \to ^{BM} H_r.$
We suggest \cite {Br} chapter 5 as reference to Borel-Moore homology.

Recall that Borel-Moore homology with coefficients in a field $\kappa$ is a collection of $\kappa-$vector space-valued  functors $ ^{BM} H_r $ defined on the category of pairs $(X,Y), Y\subseteq X,$ 
$X$ 
locally compact space, with  $Y$ closed subset of $X,$  and proper continuous maps of pairs, which are homotopic functors (i.e. proper homotopic maps induce equal linear maps) with the excision and the long exact sequence  property for any such pair $(X,Y)$. When restricted to the subcategory of pairs of compact spaces they coincide with the standard homology. The reader  should be aware that 
for  a pair $(M, N)$ of f.d. manifolds,  each interiors of compact manifold with boundary, with $N$ a closed subset,  $^{BM} H_r(M,N)= H_r(\hat M, \hat N)$ and $^{BM} H_r(M)= H_r(\hat M, \ast)$ where $\hat X$ denotes the one point compactification of the locally compact space $X.$  This  partially explains the natural maps $H_r(\cdots)\to ^{BM} H_r(\cdots)$  from the standard homology (= singular homology) to the Borel-Moore homology as well as the contravariant behavior of Borel-Moore homology when restricted to open sets of a locally compact space; both properties hold  for any $X, Y$ locally compact spaces not necessary manifolds.  

The Borel-Moore homology is not a priory defined  
for a pair $(X,Y)$  when both $X$ and $Y$ are  locally compact but $Y$ is not closed, however a vector space  
$^{BM} \mathcal H^f_r(X,Y)$ can be introduced  in the  case $Y$ is open in $X$ and appears as an open sub-level of a continuous real-valued  map $f,$ as described in  Section \ref{S3}.  
 
Note  that all definitions and concepts derived in \cite{BU1} using standard homology   can be "word by word"  repeated for Borel-Moore homology. In particular one can produce the Borel-Moore version of the vector spaces \  $\mathbb I^f_a(r), \ \mathbb I_f^b (r), \ \mathbb F^f_r(a,b), \ \mathbb T^f_r(a,b),\  \hat \delta ^f _r(a,b), \hat \gamma^f_r(a,b), \hat \lambda^f_r(a)$ and ultimately of the numbers $\delta^f_r(a,b),$  $\gamma^f_r(a,b),$ $\delta^\omega_r(t), \gamma^\omega_r(t), \lambda ^f_r(a)$ denoted by the same symbols with the addition of  the left-side exponent $\{BM\}$. 
The vector spaces $\hat \lambda^f_r(a)$ / numbers $\lambda^f_r(a)$  have not been described in \cite{BU1} ;  they are defined here in  Section \ref{S3}. For $f$ a lift of a tame map they are finite dimensional/ finite  and vanish for $a$ regular value. We believe they vanish for any $a$  being actually irrelevant  intermediate objects.  In general these spaces / numbers defined using Borel-Moore homology are different from the ones defined using standard homology.  This is reviewed in .  
The natural transformations $\theta_r: H_r(\cdots)\rightsquigarrow ^{BM} H_r(\cdots) $ induce the  natural linear maps  
 $\theta_r(a): H_r(\tilde X_a,\tilde X_{<a})\to ^{BM} \mathcal H_r(\tilde X_a, \tilde X_{<a}),$  $\theta^\delta _r(a,b) :\hat \delta^f_r(a,b)\to^{BM} \hat \delta^f_r(a,b).$ $\theta^\gamma_r(a,b) :\hat \gamma^f_r(a,b)\to ^{BM} \hat \gamma^f_r(a,b)$ and $\theta^\lambda_r(a) :\hat \lambda^f_r(a)\to ^{BM} \hat \lambda^f_r(a)$.
As already noticed in \cite {BU1} if $\hat\delta^f_r(a,b)\ne 0$  resp $\hat \gamma_r(a,b)\ne 0$ implies that both $a$ and $b$ are homological critical values w.r. to the homology under consideration. It is a consequence of Proposition \ref{PP2} below that $t$ is a critical value w.r. to standard homology iff is a critical value w.r. to Borel-Moore homology. 

Recall from \cite{BU1} that $\mathcal Z^1_t( X;\xi)$  denotes the subspace of tame topological closed one forms with the topology induced from the compact open topology  of $\mathcal Z^1(X;\xi),$ the space of topological closed one forms of cohomology class $\xi,$ as defined in Section 2 of \cite {BU1}, and  $Conf_{\beta^N_r(X:\xi)}(\mathbb R)$ 
denotes the space of configurations of points with multiplicity  in $\mathbb R$ of total cardinality $\beta^N_r(X;\xi)$ equipped with the {\it collision topology}. Here $\beta^N_r(X;\xi)$ denotes the Novikov-Betti number of $(X,\xi), \  \xi \in H^1(X;\kappa).$ Recall that the space of configurations of points in $Y$ of total cardinality $N$ with the collision topology can be identified  to the quotient space $Y^N/\Sigma_N, $  where $\Sigma_N$ denotes the group of permutations of $N$ elements acting on $Y^N$ by permutations.  

The main results verified in this paper are the following: %

\begin{theorem}\label {TP}\

For $M$ is a closed topological manifold and $\omega\in  \mathcal Z^1_t( X;\xi)$  the following holds true. 
\begin {enumerate}
\item $^{BM} \boldsymbol  \delta^\omega_r(t)= \boldsymbol  \delta^{\omega}_{n-r} (-t),$ 

\item $^{BM} \boldsymbol \gamma^\omega_r(t)= \boldsymbol \gamma^{-\omega}_{n-r-1} (t).$
\end{enumerate}
\end{theorem}
 
\begin{theorem} \label {TS}\

 The assignments     $\omega \rightsquigarrow \delta^\omega_r$  from  $\mathcal Z^1_t( X;\xi)$ to $Conf_{\beta^N_r(X:\xi)}(\mathbb R)$
  is  continuous   when the source $\mathcal Z^1_t( X;\xi)$ is equipped with the compact-open topology and  the target $Conf_{\beta^N_r(X:\xi)}(\mathbb R)$ with the collision topology.

The assignment can be extended continuously to the entire $\mathcal Z^1_t( X;\xi).$
\end{theorem}

If $f:\tilde X\to \mathbb R$ is a lift of a tame TC1-form $\omega,$  on a compact ANR $X,$ 
 cf. \cite {BU1} or Section \ref{S3}  for definition, one has 

 \begin{proposition}\label{PP1}\ 

1.\  $\dim ( H_r(\tilde X_a, \tilde X_{<a}))= \sum_{t\in \mathbb R} \delta^f_r(a,t) + \sum _{t>a}  \gamma ^f_r(a, t) + \sum  \gamma^f_{r-1} (t, a) + \dim \hat \lambda^f_{r-1}(a).$ 

2.\  $\dim ( ^{BM} \mathcal H_r(\tilde X_a, \tilde X_{<a}))= \sum_{t\in \mathbb R} \ ^{BM} \delta^f_r(a, t) + \sum _{t>a}  \ ^{BM} \gamma ^f_r(a, t) + \sum_{t<a} \  ^{BM} \gamma^f_{r-1} (t, a)+ \dim {^{BM} \hat \lambda}^f_{r-1}(a).$ 

\end{proposition} 

Formula 1.  was established  in \cite{BU1}

We will also establish the following results essential for the proofs in Theorems 1.2 and 1.3 in \cite{BU1}.

\begin{proposition} \label {PP2}\ 

The  linear maps  $\theta_r(a): H_r(\tilde X_a,\tilde X_{<a})\to ^{BM} \mathcal H_r(\tilde X_a, \tilde X_{<a})$ are  isomorphisms. 
\end{proposition} 
One expects that both $\theta^\delta _r(a,b) :\hat \delta^f_r(a,b)\to ^{BM} \hat \delta^f_r(a,b)$ and $\theta^\gamma_r(a,b) :\hat \gamma^f_r(a,b)\to ^{BM} \hat \gamma^f_r(a,b)$ are isomorphisms. 
 This is indeed the case when $\omega$ is of degree of irrationality $1.$  

Theorem \ref{TS} verifies Theorem 3 item 1 in \cite {BU1}. Item 2 will be treated in part 3 of this work. 

Theorem \ref{TP} verifies Theorem {T2} announced in \cite {BU1} provided $\theta_r^\delta$ and $\theta_r^\gamma$ are isomorphisms, fact not yet established  
 in full generality and possibly not always true. 
 Proposition (\ref {PP2}) shows however that $a$ is homological critical value for standard homology iff is homological critical value for Borel-Moore homology a first step towards verifying $\theta^\delta_r$, $\theta_r^\gamma,$ $\theta^\lambda_r$ and $\lambda^f_r$ are isomorphisms if the case.  

The paper requires some basic linear algebra, most likely familiar to most of the readers  but for convenience of the reader reviewed in Section 2. The reader can skip this section at the first reading and return to it when necessary.

\section {Linear algebra preliminary}

\subsection {Linear algebra }

In this section "$=$" designates equality or canonical isomorphism and $\simeq$ indicates the existence of an isomorphism.  

Let $\kappa$ be a fixed field. All vector spaces are $\kappa-$vector spaces and the linear maps are $\kappa-$linear.
Recall that for a linear map  $\alpha :A\to B$  the {\it canonical exact sequence  associated to $\alpha$} is 
$$
\xymatrix
{0 \ar[r]&\ker \alpha \ar[r]^{i_\alpha} &A\ar[r]^\alpha &B \ar[r]^{\pi_\alpha} &\coker\ \alpha \ar[r] &0}.
$$  
By passing to "duals"  one obtains  the  exact sequence 
$$
\xymatrix  {0&(\ker \alpha)^\ast\ar[l] &A^\ast \ar[l]_{i_\alpha^*} &B^\ast \ar[l]_{\alpha^\ast} &(\coker\ \alpha)^\ast\ar[l]_{{\pi_\alpha}^*} &0 \ar[l]},
,$$
canonically isomorphic to the canonical exact sequence of $\alpha^\ast.$  

\begin{equation}\label {E1}
\xymatrix{
0\ar[r]&\ker (\alpha^\ast) \ar[r]^{{i_\alpha}^*} &B^\ast \ar[r]^{\alpha^\ast} &A^\ast  \ar[r]^{{\pi_\alpha}^*} &(\coker \alpha^\ast) \ar[r] &0\\
0\ar[r]&(\coker \alpha)^\ast \ar[u]^\theta\ar[r]^{\pi_{\alpha}^\ast} &B^\ast\ar[u]^= \ar[r]^{\alpha^\ast} &A^\ast\ar[u]^= \ar[r]^{i_{\alpha}^*} &(\ker \alpha)\ast \ar[u]^\theta \ar[r] &0}
\end{equation}

For a diagram 
$$
D:= \xymatrix { A\ar[d]_{\alpha_2}\ar[r]^{\alpha_1} &B_1\ar[d]^{\beta_1}\\B_2\ar[r]^{\beta_2} &C}
$$
let $\underline{\alpha} : A\to \alpha_2\cup_ A \alpha_1$ be the push-forward of $\alpha_2$ and $\alpha_1,$  and $\overline \beta: \beta_2\times_C\beta_1$ 
be the pullback of $\beta_2$ and $\beta_1,$  with  
\begin{equation*} \begin{aligned} 
\alpha_2\cup_ A \alpha_1:= &B_2 \oplus B_1/ \{\alpha_2(a)-\alpha_1(a) \mid a\in A\}, \\ \underline \alpha(a)= &\langle \alpha_2(a)\oplus \alpha_1(a)\rangle, 
\\ 
\beta_1\times_ C \beta_2:=& \{ (b_1, b_2) \in B_1 \times  B_2\mid \beta_1(b_1)= \beta_2(b_2)\},\\ \overline \beta(b_1, b_2)=& \beta_1(b_1)= \beta_2(b_2) .
\end{aligned}
\end{equation*}
where $\langle b_2 \oplus b_1\rangle$ denotes the image of $b_1 \oplus b_2$ in $\alpha_2 \cup_A \alpha_1.$ 

Let     
$\alpha: A\to \beta_2\times_C \beta_1$ and 
$\beta: \alpha_2\cup A\alpha_1\to C$
be given by 
\begin{equation*}
\begin{aligned}
\alpha (a)= &(\alpha_2(a), \alpha_1(a))\\
\beta(\langle b_2\oplus  b_1\rangle)=& \beta_2(b_1) + \beta_2(b_2) .  
\end{aligned}
\end{equation*} 

Define  $$\ker (D):= \ker \alpha$$
$$\coker (D):= \coker \beta.$$

\begin{obs}\label {O2.1}The canonical isomorphisms $\theta$   extend to the canonical isomorphisms 
$$\theta(D): (\coker D)^\ast \to \ker (D^\ast) $${and} $$\theta(D): (\ker D)^\ast \to \coker (D^\ast).$$
\end{obs}

\proof : To check the statements  observe that 
the diagram $D$ can be completed to  the diagrams $\underline D$ and $\overline D$ 
\vskip .1in
\hskip .3in $
\underline D:= \xymatrix { A\ar[dd]_{\alpha_2}\ar[rd]^{\underline \alpha}\ar[rr]^{\alpha_1} &&B_1\ar[ld]^{i_1}\ar[dd]^{\beta_1}\\
&\alpha_2\cup_ A \alpha_1\ar[rd]^{ \beta}\\
B_2\ar[ru]^{i_2}\ar[rr]^{\beta_2} &&C}
$\hskip .5in
$
\overline D:= \xymatrix { A\ar[dd]_{\alpha_2}\ar[rr]^{\alpha_1}\ar[rd]^{\alpha} &&B_1\ar[dd]^{\beta_1}\\
&\beta_2\times_ C \beta_1\ar[ld]^{p_2}\ar[ru]^{p_1}\ar[rd]^{\overline \beta}\\
B_2\ar[rr]^{\beta_2} &&C}
$

and notice  that $(\underline D)^\ast $identifies to $\overline {(D^\ast)}$  and $(\overline D)^\ast $identifies to $\underline {(D^\ast)}$
 which imply the statements. 
 
 q.e.d.
 \vskip .2in

Let  $\alpha:A\to B,\  \beta:B\to C,\  \gamma:C\to D$ be linear  maps.  
To these three maps we associate the  diagrams,  $\widehat {\mathcal D}$ and $\underline {\mathcal D}$:  
$$\widehat {\mathcal D}(\alpha, \beta, \gamma)\equiv  \begin{cases}\xymatrix {\ker (\gamma \beta \alpha)\ar[r]^{j_2}&\ker (\gamma\beta)\\
\ker (\beta \alpha)\ar[r]^{j_1}\ar[u]_{i_1}\ar[ru]^k&\ker (\beta) \ar[u]^{i_2}} \end{cases}\quad \underline {\mathcal D}(\alpha, \beta, \gamma) \equiv\begin{cases} 
 \xymatrix {\coker (\gamma \beta\alpha)\ar[r]^{j'_2}&\coker (\gamma \beta )\\
\coker (\beta\alpha)\ar[r]^{j'_1}\ar[u]_{i'_1}\ar[ru]^{k'}&\coker (\beta) \ar[u]^{i_2'}}\end{cases}.$$
In the diagram  $\hat D$
\begin{enumerate} [label= (\alph*)] 
\item   
$i_1$ and $i_2$ are  injective, 
\item $i_1: \ker j_1\to \ker j_2$  is an isomorphism,
\item $ \img\  k= \img\ j_2 \cap \img\ i_2$
\end{enumerate}
and  in the  diagram $\underline D$\ \   
\begin{enumerate} [label= (\alph*)] 
\item  $j_1'$ and $j'_2$ are surjective, 
\item $\coker  i'_1\to \coker i'_2$  is an  isomorphism,   
\item $\ker (j_2'\cdot i_1' =i_2'\cdot j'1)= \ker i'_1 + \ker j_1'.$ 
\end{enumerate} 

In view of Observation \ref{O2.1}  
one has 
\begin{obs} \label {O2.2}
$( \hat D(\alpha, \beta, \gamma))^\ast= \underline {\mathcal D}(\gamma^\ast, \beta^\ast, \alpha^\ast) $
and $( \underline D(\alpha, \beta, \gamma))^*= \hat {\mathcal D}(\gamma^*, \beta^*, \alpha^*). $
\end{obs}
\vskip .1in
One defines the  vector space 
\begin{equation} \label  {E2}
\begin{aligned}
\hat {\omega}  (\alpha, \beta, \gamma)= \coker \hat {\mathcal D}(\alpha, \beta, \gamma):=  \coker (j_1 \cup _{\ker(\beta\alpha)}i_1 \to \ker (\gamma\beta))\\
= \boxed{\ker (\gamma \beta)/ (j_2 (\ker(\gamma\beta\alpha)) + i_2(\ker \beta))},
\end{aligned}
\end{equation} 
a quotient space of $\ker(\gamma \beta)$, and the vector space 
$$
\underline \omega (\alpha, \beta, \gamma)=  \ker \underline {\mathcal D}(\alpha, \beta, \gamma):= \boxed{\ker  (  \coker(\beta\alpha)\to i_2' \times _{\coker (\gamma\beta)} j'_2)},
$$
a subspace of $\coker (\beta\alpha).$
\vskip .1in 
Note  that the assignments $(\alpha, \beta, \gamma)\rightsquigarrow \widehat  \omega(\alpha, \beta, \gamma)$ and $(\alpha, \beta, \gamma)\rightsquigarrow \underline \omega(\alpha, \beta, \gamma)$
are functorial and 
in view of  the definitions above  
if $\alpha$ is surjective  or if $\gamma$ is injective then $\hat\omega(\alpha, \beta, \gamma)=0.$

 \begin{theorem}\label {T2.3}\ 
 
{\rm 1.}  For $\alpha, \beta, \gamma$  linear maps the isomorphisms $\theta$ extend to the canonical isomorphisms 
$$\theta : \hat\omega(\alpha, \beta, \gamma)^\ast  \to \underline  \omega(\gamma^\ast, \beta^\ast, \alpha^\ast)$$
$$\theta : \underline\omega(\alpha, \beta, \gamma)^\ast  \to \hat  \omega(\gamma^\ast, \beta^\ast, \alpha^\ast).$$

{\rm 2.}  For  $\alpha, \beta, \gamma$ and $\alpha', \beta', \gamma'$ linear maps consider the diagram 
\begin{equation}\label {D3}
\xymatrix{&M\ar[r]^{=} &M\ar[r]^{=}&M\ar[r]^{=}&M\\
&A\ar[r]^\alpha \ar[u]^{\lambda_A} &B\ar[r]^\beta\ar[u]^{\lambda_B}&C\ar[r]^\gamma\ar[u]^{\lambda_C}&D\ar[u]^d\\ 
&A'\ar[r]^{\alpha'}\ar[u]^a &B'\ar[r]^{\beta'}\ar[u]^b&C'\ar[r]^{\gamma'}\ar[u]^c&D'\ar[u]^d\\
&N\ar[u]^{\theta_A}\ar[r]^{=} &N\ar[u]^{\theta_B}\ar[r]^{=}&N\ar[u]^{\theta_C}\ar[r]^{=}&N\ar[u]^{\theta_D}
.}
\end{equation}
If  the columns are exact sequences then the linear maps  
$$\widehat \omega(\alpha', \beta', \gamma')\to  \widehat \omega(\alpha, \beta, \gamma) $$and 
$$\underline  \omega(\alpha', \beta', \gamma')\to \underline \omega(\alpha, \beta, \gamma) $$
are isomorphisms.

{\rm 3.} For $\alpha, \beta, \gamma$  linear maps the following holds true. 

\begin{enumerate} 
\item[{\rm (i)}] \label{E4}  
A factorization $\alpha= \alpha_2\cdot \alpha_1$, with  $\alpha_1:A\to A'$ and $\alpha_2: A'\to B$ linear maps, induces the short exact sequences 
\begin{equation}
\begin{aligned}
0 \to \widehat \omega(\alpha_1, \beta\alpha_2, \gamma) \to \widehat \omega(\alpha, \beta,\gamma) \to \widehat \omega (\alpha_2, \beta, \gamma)\to 0,\\
0 \to \underline \omega(\alpha_1, \beta\alpha_2, \gamma) \to \underline \omega(\alpha, \beta,\gamma) \to \underline  \omega (\alpha_2, \beta, \gamma)\to 0.
\end{aligned}
\end{equation}
\item[{\rm (ii)}]  \label {E5}A factorization $\beta= \beta_2 \cdot \beta_1$, with $\beta_1:B\to B'$ and $\beta_2: B'\to B$ linear maps, induces the short exact sequences 
\begin{equation}
\begin{aligned}
0 \to \widehat \omega( \alpha, \beta_1,\beta_2) \to \widehat \omega(\alpha, \beta_1,\gamma\beta_2)\to \widehat \omega (\alpha, \beta,\gamma)\to 0,\\
0 \to \underline \omega( \alpha, \beta_1,\beta_2) \to \underline \omega(\alpha, \beta_1,\gamma\beta_2)\to \underline \omega (\alpha, \beta,\gamma)\to 0.
\end{aligned}
\end{equation}
\item[{\rm (iii)}]  \label {E6}A factorization $\gamma= \gamma_2 \cdot \gamma_1 $, with  $\gamma_1:C\to C'$ and $\gamma_2: C'\to D$  maps, induces the short exact sequences 
\begin{equation}
\begin{aligned}
0 \to \widehat \omega( \alpha, \beta, \gamma_1) \to \widehat \omega(\alpha, \beta,\gamma)\to \widehat \omega (\alpha, \gamma_1\beta,\gamma_2)\to 0,\\
0 \to \underline \omega( \alpha, \beta, \gamma_1) \to \underline \omega(\alpha, \beta,\gamma)\to \underline \omega (\alpha, \gamma_1\beta,\gamma_2)\to 0.
\end{aligned}
\end{equation}
\end{enumerate}
\end{theorem}

\begin{proof}
 Item 1. follows from Observations (\ref{O2.1}) and (\ref{O2.2}). 

To prove Item 2.  notice the following.
\begin{enumerate}
[label=(\alph*)]
\item
If $N=0$, then the induced maps
\begin{equation*}
\begin{aligned}
\ker (\beta')\to \ker( \beta),\\
\ker (\gamma' \beta')\to \ker(\gamma \beta)\\
\ker (\beta'\alpha' )\to \ker(\beta\alpha),\\
\ker (\gamma' \beta' \alpha' )\to \ker (\gamma \alpha \beta).
 \end{aligned}
 \end{equation*} 
are isomorphisms hence
$\hat D(\alpha', \beta', \gamma')=\hat D(\alpha, \beta, \gamma) $ and then $\hat \omega(\alpha', \beta', \gamma')=  \hat \omega(\alpha, \beta, \gamma).$

\item 
If $M=0$ then  the obvious induced maps 
\begin{equation*}
\begin{aligned}
\coker(\beta' \alpha')\to \coker(\beta \alpha)\\
\coker(\gamma' \beta' \alpha')\to \coker(\gamma \beta \alpha)\\
\coker(\gamma' \beta' )\to \coker(\gamma \beta )\\
\coker(\beta' )\to \coker( \beta )\\
\end{aligned}
\end{equation*}
are isomorphisms 
consequently, the induced map $\hat\omega(\alpha', \beta' \gamma')\to  \hat \omega(\alpha, \beta, \gamma)$ is an isomorphism.  

\item To prove the result for $M$ and $N$ arbitrary consider the diagram 
$$
\xymatrix{
&A\ar[r]^\alpha &B\ar[r]^\beta&C\ar[r]^\gamma&D\\
&\img \ a \ar[r]^{\alpha''}\ar[u]&\img\ b\ar[r]^{\beta''}\ar[u] &\img\ c \ar[r]^{\gamma''}\ar[u] & \img \ d \ar[u]
\\ &A'\ar[r]^{\alpha'}\ar[u] &B'\ar[r]^{\beta'}\ar[u]&C'\ar[r]^{\gamma'}\ar[u]&D' \ar[u]
}$$
and the linear maps 
$$
\widehat \omega(\alpha', \beta', \gamma')\to\widehat \omega(\alpha'', \beta'', \gamma'')\to \widehat \omega(\alpha, \beta, \gamma).
$$ 
The first arrow is an isomorphism by (b) above and the second by (ai) above.  This establishes Item 2.
\end{enumerate}      
To prove Item 3  consider  the diagram 

 \begin{equation}\label {D7}
\xymatrix{A_2\ar[r]^{i^A_2}\ar@/^2pc/[rr]_{i_2}&B_2\ar[r]^{i^B_2}&C_2\\
A_1\ar[u]_{j_A}\ar@/_2pc/[rr]^{i_1}\ar[ur] \ar[r]_{i^A_1} \
\ar[urr]& B_1\ar[u]_{j_B}\ar[ur]\ar[r]_{i^B_1}&C_1\ar[u]_{j_C}}
\end{equation}
and  make the following observation.

\begin {obs} \label{O2.4}
Suppose that each of the three  diagrams $\mathbb B_1, \mathbb B_2, \mathbb B,$ associated with  (\ref{D7}),     

$\mathbb B_1$ with vertices $A_1, A_2, B_1, B_2$,

$\mathbb B_2$ with vertices $B_1, B_2, C_1, C_2$
 and 

$\mathbb B$ with vertices $A_1, A_2, C_1, C_2$

\noindent satisfy the properties (a) (b) (c) of the diagram $\hat D.$\  
Then (\ref{D7})  induces  the  exact sequence  
\begin{equation*}
\xymatrix@C-5pc { B_2/ (i^A_2(A_2)+ j_B(B_1))\ar[rd]^{i}&&  \\
0\ar[u]& C_2/ (i_2(A_2)+ j_C(C_1)) \ar[rd]^{p}&0\\ 
&&C_2/ (i^B_2 (B_2)+ j_C(C_1))\ar[u] }
\end{equation*}
with  $i$ induced  by $i^B_2,$  
well defined because $\img ( i^B_2\cdot j_B)\subseteq \img j_C$,
 and $p$ the projection induced by the inclusion $(i_2(A_2)+ j_C(C_1)) \subseteq (i^B_2 (B_2)+ j_C(C_1)).$
\vskip .1in 
\end{obs}
Clearly  $p$ is surjective and $p\cdot i=0.$  Property (c) 
implies that $i$ is injective. Properties (a), (b) (c) imply that  the sequence is exact.

A similar observation holds for the diagram 
\begin{equation}\label {D8}
\xymatrix {A_3\ar[r]^{i_3}& B_3\\ 
A_2\ar[r]^{i_2}\ar[ru]\ar[u]^{j_2^A}&B_2\ar[u]_{j_2^B}\\
A_1\ar@/^2pc/[uu]^{j^A}\ar[u]^{j_1^A}\ar[ur]\ar[ruu]\ar[r]^{i_1}&B_1\ar@/_2pc/[uu]_{j^B}\ar[u]_{j_1^B}.
}
\end{equation}

\begin{obs} \label {O2.5}
Suppose that each of the three  diagrams $\mathbb B_1, \mathbb B_2, \mathbb B,$  associated with  (\ref{D8}),    

 $\mathbb B_1$ with vertices $A_2, A_3, B_2, B_3$,

$\mathbb B_2$ with vertices $A_1, A_2, B_1, B_2$ and 

$\mathbb B$ with vertices $A_1, A_3, B_1, B_2$

\noindent satisfy the properties (a) (b) (c) of the diagram $\hat D$ 
Then (\ref{D8})  induces  the  exact sequence  
\begin{equation*}
\xymatrix @C-5pc{B_2/ (i_2(A_2)+ j^B_1(B_1))\ar[rd]^{j}&&  \\
0\ar[u]& B_3/ (i_3(A_3)+ j^B(B_1)) \ar[rd]^{p}&0\\ 
&&B_3/ (i_3 (A_3)+ j_2^B(B_2))\ar[u].}
\end{equation*}
\end{obs}

Observation \ref{O2.4} applied to  diagram  (\ref{D7}) with
\begin{equation*} 
\begin{aligned}
&A_1= \ker(\beta\alpha),                &A_2=& \ker(\gamma\beta\alpha),\\
&B_1= \ker(\beta\alpha_2),            &B_2=&\ker(\gamma\beta\alpha_2),\\
&C_1= \ker(\beta),                          &C_2=&\ker(\gamma\beta) 
\end{aligned}
\end{equation*} verifies Item 3. (i). 
 
Observation \ref{O2.5} applied to  diagram  (\ref{D8}) with 
\begin{equation*} 
\begin{aligned}
&A_1= \ker(\beta_1\alpha),             &B_1=& \ker(\beta_1),\\
&A_2= \ker \beta\alpha),                 &B_2= &\ker (\beta),\\
&A_3= \ker(\gamma \beta\alpha),   &B_3= &\ker(\gamma\beta) 
\end{aligned}
\end{equation*} verifies Item 3. (ii) 
and applied to  diagram  (\ref{D8})  with
\begin{equation*} 
\begin{aligned}
&A_1= \ker(\beta\alpha),                   & B_1= &\ker(\beta),\\
&A_2= \ker(\gamma_1\beta\alpha),  & B_2= &\ker(\gamma_1\beta),\\
&A_3= \ker(\gamma\beta\alpha),      & B_3= &\ker(\gamma\beta) 
\end{aligned}
\end{equation*} verifies Item 3. (iii). 

\end{proof}

The above considerations/proofs  were already contained in \cite {B} but under the additional hypothesis that all the linear maps were  Fredholm. 
\vskip .2in 

{\bf Direct and inverse limits}

 A totally order set $\mathcal I= (I, \leq)$ can be  viewed as a category whose objects are its elements and for any two objects $x,y \in I$ there is only one morphism if $x\leq y$ and no morphism otherwise. 
  A system $\mathcal A$ indexed by $I$ is a covariant functor from $\mathcal I$ to the category of $\kappa-$vector spaces, i.e the collection of vector spaces $ \{A_t, t\in I\}$ and linear maps $ i_t^s: A_t\to A_s, t\leq s\}.$ 
Each such system has a direct limit $\varinjlim_{t\in I} A_t$ and inverse limit $\varprojlim_{t\in I} A_t$ both vector spaces.  A priory there is also a derived  direct limit but always identical to zero and a derived inverse limit ${\varprojlim }'_{t\in I} A_t$ a vector space not always trivial, 
  as well as the obvious linear maps $ \pi_t, i_t, i $ making the diagram below commutative for any $t\leq<s$
  $$\xymatrix{ \varprojlim_{t\in I} A_t \ar[rr]^i \ar[rd]^{p_t} \ar[rdd]^{p_s}&&\varinjlim_{t\in I} A_t \\
&A_t\ar[d]^{i_t^s}\ar[ru]_{i_t}&\\
&A_s\ar[ruu]_{i_s}& }.$$
 Note that  a subset $I'$ of $I$  it is called cofinal w.r.to  $I$ to the left  resp. to the right if for any $t\in I'$  there exists $ s\in I$ s.t. $s\leq t$ resp $\geq t$ and if so  $\varprojlim_{t\in I'} A_t = \varprojlim_{t\in I} A_t $ resp. $\varinjlim_{t\in I'} A_t = \varinjlim_{t\in I} A_t .$
  
 Note that $\mathbb Z_{\leq N}$ is cofinal to $\mathbb R$ to the left  (as well to $\mathbb Z$), $\mathbb Z_{\geq N}$ is cofinal to $\mathbb R$to the right (as well to $\mathbb Z$) and $\mathbb Z$ is cofinal to $\mathbb R$  both to the left and right.  

\begin{obs}\label {O2.6} \
  
 1.  An exact sequence of  systems indexed by the same $I,$  \  
$ \cdots  \mathcal A\to \mathcal B \to \mathcal C \to \mathcal  D \to \cdots .$

 induces by passing to direct limits an exact sequence
  $$ \cdots  \varinjlim  A_t\to \varinjlim  B_t\to \varinjlim  C_t\to\varinjlim  D_\to \cdots .$$ 
  This is not true for inverse limits however one has.
  
  2. A short exact sequence of  systems  indexed by the same $I,$  \
$ 0\to   \mathcal A \to \mathcal B \to \mathcal C\to 0$
  
  induces by passing to the inverse  limits the exact sequence
  $$ 0\to   \varprojlim  A_t\to \varprojlim B_t\to \varprojlim  C_t\to  {\varprojlim} ' A_t \to 0 .$$  
   \end{obs}  
  \vskip .1in

As indicated in \cite {Mi} 
 \begin{enumerate}
\item For the set $\mathbb Z_\geq N$ 
consider the linear map 

$\mathcal I: \oplus_{k\geq N}  A_k\to  \oplus_{k\geq N}$
defined by $$\mathcal I( (a_N, a_{N+1}, \ a_{N+2}, \ \cdots )= (a_N, \ a_{N+1} - i_N ^{N+1}(a_N), \ a_{N+2} - i_{N+1} ^{N+2}(a_{N+1}),\ \cdots).$$ Then 
$$\varinjlim   A_k = \coker  \ \mathcal I$$ 
\item For the set $\mathbb Z_{\leq N}$ 
consider the linear map $\mathcal P: \prod _{k\leq N}  A_k\to  \prod_{k\leq N}$
defined by $$\mathcal P (  a_{-N}, \ a_{-N-1}, \ a_{-N-2}, \cdots )= (a_N-i_{-N-1}^{-N},\  a_{-N-1} - i_{-N-2} ^{-N-1}(a_{N-2}), \ a_{-N-2} - i_{-N-3}^{-N-2}( a_{-N-3}), \cdots  ).$$ Then          
$$\varprojlim   A_k = \ker   \mathcal P , \quad \ 
 {\varprojlim} '   A_k =coker   \mathcal P$$ 
 and note that   $\varprojlim'   A_k=0$ if  there exists $l$ s.t for $j$ large enough $i_{-j-l}^{-j}$ is surjective, i.e Mittag- Leffler condition is satisfied. 
\end{enumerate}
 \vskip .2in  

For the proof of Theorem \ref{TP} 
 one needs some additional definitions.

\begin {definition} \label {Def2.71}\

 A sub-surjection $\tilde \pi: A \rightsquigarrow A'$ is a pair $\tilde \pi:= \{\pi: A\to P, P\supseteq A'\}$ consisting of a surjective linear map $\pi$ and a subspace $A'$ of $P.$
\end{definition}

To  a directed system of sub-surjections 
$$
\xymatrix{  A _1\ar@{~{>}}[r] ^{\tilde \pi_1}& A_2 \ar@{~{>}}[r]^{\tilde \pi_2} &A_3\ar@{~{>}}[r]^{\tilde \pi_3}&\cdots   \ar@{~{>}}[r]^{\tilde \pi_{k-1}} &A_k\ar@{~{>}}[r]^{\tilde \pi_k} &A_{k+1}\ar@{~{>}}[r]^{\tilde \pi_{k+1}}&\cdots}
$$
one provides  a unique maximal directed system of surjections
$$
\xymatrix{  
A_1                          & A_2                                           &A_3                                         &                                                            &A_k                                            &A_{k+1}\\
A_1^\infty\ar@{->>}[r] ^{ \pi^\infty_1}\ar[u]^{\subseteq}& A^\infty_2 \ar@{->>}[r]^{ \pi^\infty_2}\ar[u]^{\subseteq} &A^\infty_3\ar@{->>}[r]^{ \pi^\infty_3}\ar[u]^{\subseteq}&\cdots   \ar@{->>}[r]^{ \pi^\infty_{k-1}} &A^\infty_k\ar@{->>}[r]^{ \pi^\infty_k}\ar[u]^{\subseteq} &A^\infty_{k+1}\ar@{->>}[r]^{ \pi^\infty_{k+1}\ar[u]^{\subseteq}}\ar[u]^{\subseteq}&\cdots}
$$
The construction of this sequence is rather straightforward  and it goes as follows:

Starting with the right side of the Diagram (\ref{D10}) below  (i.e. the collections  of linear maps $\{\pi_i, \supseteq  \},$ one inductively, from right to left  and from up to down (i.e  from  the lower index $i$ to lower index  $(i-1)$ and from the upper index $k$ to the upper index $(k+1)$)  one produces  the subspace $A_i^{k+1} \subset A^k_i$ and the surjective linear maps $\pi_i ^k: A_i^k\to A_{i+1}^k$
defined by: 

$A_k^{k+1}= (\pi^k)^{-1}(A_{k+1}),$ 

$A_k^{k+1}= (\pi^k)^{-1}(A_{k+1})$  and 

$\pi^{k+1}_i =$ restriction of $\pi^k_i.$ 

\noindent  Take $A^\infty _k = \cap_{i>k} A^k_i$ and $\pi_i^\infty$ the restriction of $\pi_i^{\cdots}$  and define 

\begin{equation} \label {E9}
\varinjlim _{i\to \infty}  {\tilde \pi_i:} = \varinjlim _{i\to \infty} {\pi^\infty_i}.
\end{equation}

\begin{equation}\label {D10}
\scriptsize 
\xymatrix
{
A_1\ar[r]^{\pi_1} &P_1 & & & & & & \\
A^2_1\ar[r]\ar[u]^\supseteq &A_2\ar[r]^{\pi_2}\ar[u]^\supseteq &P_2 & & & & & \\
A^3_1\ar[r]\ar[u]^\supseteq &A^3_2\ar[r]\ar[u]^\supseteq &  A_3\ar[r]^{\pi_3}\ar[u]^\supseteq &P_3 & & & & \\
\cdots\ar[u] &\cdots\ar[u] &\cdots\ar[u]&\cdots\ar[u] &&\cdots&\\
A_1^k\ar[r]^{\pi_1^{k}}\ar[u]^\supseteq &A^k_2  \ar[r] ^{\pi_2^{k}}\ar[u]^\supseteq &A^k_3\ar[r]^{\pi_3^{k}} \ar[u]^\supseteq &A^k_4\ar[r] \ar[u]^\supseteq&\cdots \ar[r] &A_k\ar[r]^{\pi_k}\ar[u]^\supseteq & P_k &\\
A_1^{k+1}\ar[r]^{\pi_1^{k+1}}\ar[u]^\supseteq &A^{k+1}_2 \ar[r]^{\pi_2^{k+1}}\ar[u]^\supseteq & A^{k+1}_3\ar[r] \ar[u]^\supseteq &A^{k+1}_4\ar[r] \ar[u]^\supseteq& \cdots \ar[r]& A^{k+1}_k\ar[r]\ar[u]^\supseteq  & A_{k+1} \ar[r]^{\pi_{k+1}}\ar[u]^\supseteq&P_{k+1} \\
\cdots\ar[u] &\cdots\ar[u]&\cdots\ar[u]&\cdots\ar[u]&&\cdots\ar[u]&\cdots\ar[u]&\cdots\ar[u]\\
A_1^{\infty}\ar[r]\ar[u]^\supseteq &A^\infty_2 \ar[r]\ar[u]^\supseteq &A_3^\infty\ar[r]\ar[u]^\supseteq  &A_4^\infty\ar[r]\ar[u]^\supseteq  &\cdots \ar[r]  &A_k^\infty \ar[r]\ar[u]^\supseteq &A^\infty_{k+1}\ar[r]\ar[u]^\supseteq &\cdots    
}
\end{equation} 

\section{Notations and definitions} \label {S3}

Most of the definitions and notations below are  the same as in \cite{BU1}  where they were  considered  for the standard (=singular) homology only.  However,  they can be considered  for any homology theory with coefficients in a fixed field $\kappa,$
in particular for 
Borel-Moore homology 
of interest in this paper. 
In this paper a homology theory  is 
a collections of $\kappa-$vector space-valued covariant homotopy functors denoted  by $H_r(\cdots)$ defined on the category of pairs of locally compact Hausdorff spaces  $(X,Y),$  $Y$ closed subset of $X,$ and of proper continuous maps which satisfy the Eilenberg-Steenrod axioms.  Standard homology is defined for any pair $(X,Y)$ and in addition satisfies Milnor continuity axiom.   Recall that the {\it continuity axiom} states that  if $ X(0) \subset \cdots\subset  X(i) \subseteq X(i+1)\subseteq \cdots \subset X$ is a filtration of an  ANR $X$  by the ANRs $X(i)$  with $X= \cup_i X(i),$ then $$H_r(X)= \varinjlim_i H_r(X_i).$$ 
In this section any  homology theory, in particular  the standard  or the Borel-Moore homology, will be denoted by $H_r.$ In the next sections the notation $H_r$ will be reserved exclusively for the standard homology and the Borel-Moore homology will acquire the left-side exponent "BM" (e.g.  $^{BM} H_r$ ).

As in \cite{BU1}, for $X$ a compact ANR,  a cohomology class $[\omega]\in H^1(X;\mathbb R)$ determines the group $\Gamma:= \img ([\omega]: H_1(X;\mathbb Z)\to \mathbb R)$ and the principal $\Gamma-$covering  $\tilde X\to X$ with $\tilde X$ a locally compact ANR. As indicated in \cite{BU1} the TC1-form  $\omega$ representing $[\omega]$ determines and is determined by a $\Gamma-$ continuous equivariant map $f:\tilde X \to \mathbb R,$ unique up to an additive constant, referred to as a lift of $\omega.$ Such a map up to an equivalence (two such maps are equivalent iff their difference is a locally constant map) can be taken as an alternative definition for a TC-1 form. 

In consistency with the notation in \cite{BU1}, for a continuous map $f:Y\to \mathbb R,$ one denotes by $Y_t:= f^{-1}((-\infty, t]),$ $Y^t:= f^{-1}([t, \infty)),$  $Y(t)= f^{-1}(t)$
resp. $Y_{<t}:= f^{-1}((-\infty, t)),$ $Y^{>t}:= f^{-1} ( (t,\infty)),$ which are closed resp. open subsets of $Y.$

Note that if one specifies $f$ in the notation above, precisely  if one writes 
$Y_a^f, Y^f_{<a}, Y^a_f, Y^{>a}_f$ instead of $Y_a, Y_{<a}, Y^a, Y^{>a},$ then 
$$Y^f_a= Y^{-a}_{-f}\ \rm{and}\ Y^f_{<a}= Y^{> -a}_{-f}.$$ 

Since $\tilde X_{<a}$ is an open set,  $H_r(\tilde X_a, \tilde X_{<a})$ is not a priory defined \footnote {for example for Borel Moore homology} , but we introduce the vector space $$\mathcal H^f_r(\tilde X_a, \tilde X_{<a}):= \varinjlim_{\epsilon \to 0} H_r(\tilde X_a, \tilde X_{a-\epsilon}).$$ 
By passing to direct limit when $\epsilon\to 0$ the long exact sequence for the pair $(\tilde X_a, \tilde X_{a-\epsilon})$  leads to the long exact sequence 

\begin{equation}\label {E11}\xymatrix{ \cdots\ar[r] &  \varinjlim_{\epsilon\to 0}H_r(\tilde X_{a-\epsilon})\ar[r] & H_r(\tilde X_a)\ar[r] & \mathcal H^f_r(\tilde X_a, \tilde X_{<a})\ar[r] &  \varinjlim_{\epsilon\to 0}H_{r-1}(\tilde X_{a-\epsilon})\ar[r] &\cdots} .
\end{equation}
Similarly one introduces the vector space $$\mathcal H^f_r(\tilde X^a, \tilde X^{>a}):= \varinjlim_{\epsilon \to 0} H_r(\tilde X^a, \tilde X^{a+\epsilon}) \footnote{ for standard homology $\varinjlim_{\epsilon \to 0}  H_r(\tilde X_{a-\epsilon})  \simeq H_r(X_{<a}$ and $\mathcal H^f_r(\tilde X^a, \tilde X^{>a})\simeq H^f_r(\tilde X^a, \tilde X^{>a})$}.$$
and by passing to direct limit when $\epsilon\to 0$ the long exact sequence for the pair $(\tilde X^a, \tilde X^{a+\epsilon})$  leads to the long exact sequence 

\begin{equation}\label {E12}\xymatrix{ \cdots\ar[r] &   \varinjlim_{\epsilon\to 0}H_r(\tilde X^{a+\epsilon})\ar[r] & H_r(\tilde X^a)\ar[r] & \mathcal H^f_r(\tilde X^a, \tilde X^{>a})\ar[r] &  \varinjlim_{\epsilon\to 0}H_{r-1}(\tilde X^{a+\epsilon})\ar[r] &\cdots} .
\end{equation}

For  $f:Y\to \mathbb R$ a continuous map 
\begin {enumerate}[label= (\alph*)]
\item $t\in \mathbb R$ is a {\it  regular value} iff for any $r\geq 0$ and any open set $U\subseteq Y$ one has $h_r(U_t, U_{<t})=h_r(U^t, U^{>t})=0 $ with $h_r$ denoting the singular homology.
If $Y$ is  a smooth manifold, possibly with boundary $\partial Y,$  and $f$ is a smooth map  with $t$ a regular value in the sense of differential calculus for both $Y$ and $\partial Y,$  then $t$ is a regular value in the sense mentioned above. 
\item $t\in \mathbb R$ is a {\it  critical value} if not regular. Denote by $CR(f)\subset \mathbb R$ the set of all critical values. 
\item $x\in X$ is a critical point if for some small enough neighborhood $U$ of $x,$  $f(x)$ is a critical value for the restriction of $f$ to $U.$ 
Denote by $Cr(f)\subset X$ the set of critical points.
\end{enumerate} 

The  continuous map $f:Y\to \mathbb R$ is {\it tame} if the following holds true: 
\begin{enumerate}[label=(\roman*)] 
\item $Y$ is a locally compact  ANR and  $f^{-1}(I)$ is an ANR for any closed interval $I\subset \mathbb R,$ 
\item $CR(f)$ is countable, hence the set of regular values is dense in $\mathbb R, $  
\item for any $t\in CR(f))$ the set $f^{-1} (t)\cap Cr(f)$ is a compact ANR. 
\end{enumerate} 
Note that $Cr(f)= C_r(-f)$ and $CR(f)= - CR(-f)$ for 
$f$ a lift of the tame TC1-form $\omega$  resp. $-f$ lift of the tame TC1-form $-\omega.$

 A TC1-form $\omega$ on the compact ANR $X$ is {\it tame} if one lift $f:\tilde X\to \mathbb R,$ and then any other lift, is a tame map and in addition the set of orbits of the free action of 
 $\Gamma$ on $CR(f)$ is finite. 
 
 A smooth closed one form $\omega$ 
 on a closed smooth manifold  with the property that for any zero one can find local coordinates s.t. the components of $\omega$ are polynomials, in particular  any Morse closed one form,  is tame.   
\vskip .1in 
As in  \cite{BU1},    
for  $f:\tilde X\to \mathbb R$ and $a\in \mathbb R$ 
one defines 
\begin{itemize}
\item  $\mathbb I^f_a(r) := \img (H_r (\tilde X_a)\to  H_r(\tilde X)),$
 
 $\mathbb I^f_{<a}(r) :=   \varinjlim_{\epsilon\to 0} \img (H_r (\tilde X_{a-\epsilon})\to  H_r(\tilde X)) ,$

$\mathbb I_f^a(r) := \img (H_r (\tilde X^a)\to  H_r(\tilde X)),$ 

 $\mathbb I_f^{>a}(r): =  \varinjlim_{\epsilon\to 0} \img (H_r (\tilde X^{a+\epsilon})\to  H_r(\tilde X)), $\} \footnote {for standard homology  $\mathbb I^f_{<a}(r)\simeq \img (H_r (\tilde X_{<a})\to  H_r(\tilde X))$   and $\mathbb I_f^{>a}(r)\simeq \img (H_r (\tilde X^{>a})\to  H_r(\tilde X))$}

$\mathbb I^f_r(a',a):= \mathbb I^f_a(r)/ \mathbb I^f_{a'}(r)$ for $a' <a,$ \quad $\mathbb I^f_r(<a',a):= \mathbb I^f_a(r)/ \mathbb I^f_{<a'}(r)$\  for $a' \leq a,
$
\item $\mathbb F^f_r(a,b):= \mathbb I^f_a(r) \cap \mathbb I_f^b(r),$ \ \  $\mathbb F^f_r(<a,b):= \mathbb I^f_{<a}(r) \cap \mathbb I_f^b(r),$\ \  $\mathbb F^f_r(a, >b):= \mathbb I^f_a(r) \cap \mathbb I_f^{>b}(r),$ 
\item  $\mathbb G^f_r(a,b):= H_r(\tilde X)/ ( \mathbb I^f_a(r) +\mathbb I_f^b(r)),$ \ $\mathbb G^f_r(<a,b):= H_r(\tilde X)/ ( \mathbb I^f_{<a}(r) +\mathbb I_f^b(r)),$\  $\mathbb G^f_r(a, >b):= H_r(\tilde X)/ ( \mathbb I^f_{a}(r) +\mathbb I_f^{>b}(r)),$

In order to lighten the notation, when implicit from the context, $f$  might be dropped off the notation.

\noindent For a {\bf box} $B\subset \mathbb R^2,$  $B= (a',a]\times [b, b'),$ $a'<a, b <b', $ consider the diagrams $ \mathcal F_r(B)$ and  $\mathcal G_r(B)$  
$$\mathcal F_r(B):= \begin{cases} \xymatrix{ \mathbb F_r(a',b')\ar[d]^{\subseteq} \ar[r]^{\subseteq} &\mathbb F_r(a,b')\ar[d]^{\subseteq} \\ \mathbb F_r(a',b)\ar[r]^{\subseteq}  & \mathbb F_r(a,b)}\end{cases},\quad  
\mathcal G_r(B):= \begin{cases} \xymatrix{ \mathbb G_r(a',b')\ar@{->>}[d]\ar@{->>}[r]&\mathbb G_r(a,b')\ar@{->>}[d]\\ \mathbb G_r(a',b)\ar@{->>}[r]
& G_r(a,b)}
\end{cases}$$
whose arrows  are the obviously induced linear maps. In the diagram  $\mathcal F_r(B)$ all these induced maps are injective  and in the diagram $\mathcal G_r(B)$ all  are surjective. 

As in \cite{BU1} one defines   
\item $\mathbb F_r(B):= \coker (\mathcal F(B))= \frac {\mathbb F_r(a,b)}{ \mathbb F_r(a',b) + \mathbb F_r(a, b')},$  
\item $\mathbb G_r(B):= \ker \mathcal G_r(B)= \ker (G_r(\alpha, \beta)\to G_r(\alpha, \beta')\times_{ G_r(\alpha' ,\beta')}G_r(\alpha', \beta).$

The inclusion 
$\mathbb I_a(r)\cap \mathbb I^b(r)  \subset (\mathbb I_{a'}(r)+\mathbb I^b(r))  \cap (\mathbb I^{b'}(r) + \mathbb I_a(r))$ induces a canonical  isomorphism 
\begin{equation} \label {E13} \boxed{ \theta_r(B) : \mathbb F_r(B) \to \mathbb G_r(B)}.\end{equation} 
\end{itemize}

For $a''< a' <a$ and $b <b' <b''$  consider  the boxes $B_1, B_2, B$ with $B=(a'',a]\times [b,b'')$
and either $B_1 = (a'',a']\times [ b,b''), B_2= (a',a]\times [b,b'')$ or $B_1=(a'',a]\times [b',b''), B_2= (a'',a]\times [b,b').$  
In both cases $B= B_1\sqcup B_2.$ As in \cite{BU1} or \cite {B} one has the following.
 
\begin{proposition} (cf. \cite {BU1}) \label {P3.1}\
The boxes $B_1, B_2, B$ described above induce the commutative diagram whose rows are  exact sequences and vertical arrows are isomorphisms.
$$\xymatrix { 0\ar[r]&\mathbb F_r(B_1) \ar[r]^{i_{B_1} ^{B_2}} \ar[d]^{\theta_r(B_1)}&\mathbb F_r(B) \ar[r]^{\pi_{B_1} ^{B_2}}\ar[d]^{\theta_r(B)}&\mathbb F_r(B_2)\ar[r]\ar[d]^{\theta_r(B_2)} &0\\
 0\ar[r]&\mathbb G_r(B_1) \ar[r]^{i_{B_1} ^{B_2}} &\mathbb G_r(B) \ar[r]^{\pi_{B_1} ^{B_2}}&\mathbb G_r(B_2)\ar[r] &0 .}$$
\end{proposition}. 

\begin{obs} \label {O3.2}
As a consequence if $B_1\subset B$ are boxes, with $B_1$ located in the upper-left  corner of $B.$ then the induced linear map $i_{B_1}^{B}: \mathbb F_r(B_1)\to \mathbb F_r(B)$ is injective and, if $B_2$ is located in the down-right
corner of $B,$ then the induced linear map $\pi_{B}^{B_2}: \mathbb F_r(B)\to \mathbb F_r(B_2)$ is surjective. 
\end{obs}
 In Figure 1 below  $B_{11} = (a'',a']\times [b', b'')$ is located in the  upper-left corner of $B= (a'',a]\times [b, b'')$ and $B_{22}= (a',a]\times[b, b')$ in the lower-right corner of $B.$
  
\hskip 1 in \begin{tikzpicture} [scale=1.2]
\draw [<->]  (0,4) -- (0,0) -- (5,0);
\node at (-.2,3.5) {$b''$};
\node at (-0.2,1.5) {$b'$};
\node at (-0.2,1) {$b$};
\node at (1,-0.2) {$a''$};
\node at (3,-0.2) {$a'$};
\node at (4,-0.2) {$a$};
\draw [line width=0.10cm] (1,1) -- (4,1);
\draw [line width=0.10cm] (4,1) -- (4,3.5);
\draw [dashed, ultra thick] (1,1) -- (1,3.5);
\draw [dashed, ultra thick] (1,3.5) -- (4,3.5);
\draw [line width=0.10cm] (1,1.5) -- (4,1.5);
\draw [line width=0.10cm] (3,1) -- (3,3.5);
\node at (2,1.25) {$B_{21}$};
\node at (2,2.5) {$B_{11}$};
\node at (3.5,1.25) {$B_{22}$};
\node at (3.5,2.5) {$B_{12}$};
\node at(2.4, -1) {Figure 1
};
\end{tikzpicture}
\vskip .1in

Define $$\boxed{^F \hat \delta_r(a,b) :=\varinjlim_{\epsilon, \epsilon' \to 0} \mathbb F_r((a-\epsilon,a]\times [b, b+\epsilon')) }\ \rm {and} \  
\boxed{^G \hat \delta_r(a,b) :=\varinjlim_{\epsilon, \epsilon' \to 0} \mathbb G_r((a-\epsilon,a]\times[b, b+\epsilon'))}.$$

In view of (\ref{E13})
one has 
\begin{equation*} \label {E14}
^F {\hat \delta }_r(a,b) = ^G \hat \delta_r(a,b),  
\end{equation*} and then simplify the notation to 
 $$ \boxed{\hat \delta _r(a,b):=  ^F {\hat \delta }_r(a,b) = ^G \hat \delta_r(a,b)}$$ and denote $$   
\boxed{ \delta_r(a,b):= \dim \hat \delta_r(a,b)}.$$
 Note also that  
 \begin{equation} \label {E14}
 \hat \delta^f_r(a,b)= 
 \frac {\mathbb F^f_r(a,b) }{  \mathbb F^f_r(<a,b) + \mathbb F^f_r(a,>b)} \end{equation} 
and in view of (\ref {E14}) one has  the obvious surjective linear map $$\boxed{\pi^\delta_{a,b}(r) :\mathbb F^f_r(a,b)\to \hat \delta^f_r(a,b)}  .$$ 
\vskip .2in

As in \cite{BU1},  for  $f:\tilde X \to \mathbb R$ a tame map and $a,b\in \mathbb R\cup \infty$ with $a<b\leq \infty,$ let
$i_a^b(r): H_r(\tilde X_a) \to H_r(\tilde X_b),$ with $\tilde X_\infty: = \tilde X,$ be
the inclusion induced linear map
and define  
\begin{itemize} 
\item  
 $\mathbb T_r(a,b): =\ker(i_a^b(r): H_r(\tilde X_a)\to H_r(\tilde X_b)),$
 \item  
$\mathbb C_r(a,b): =\coker( i_a^b(r):H_r(\tilde X_a)\to H_r(\tilde X_b)).$  
\end{itemize}
and then 
\begin{itemize}
\item  
 $\mathbb T_r(<a,b): =\varinjlim_{\epsilon \to 0} \mathbb T_r(a-\epsilon, b),$
\item  
 $\mathbb T_r(a,<b): =\varinjlim_{\epsilon \to 0} \mathbb T_r(a, b-\epsilon),$ $0<\epsilon <b-a$
 \item  
$\mathbb C_r(<a,b): =\varinjlim_{\epsilon \to 0} \mathbb C_r(a-\epsilon, b).$  
\end{itemize}

For $a' <b \leq b' <b$ \ denote by \   $i^{a,b}_{a' b'}(r): \mathbb T_r(a',b')\to \mathbb T_r(a,b)$  the induced linear map.
\vskip .1in

For a {\bf box above diagonal} $ B= (a',a]\times (b',b]$  with  $a'<a \leq b' <b\leq \infty$ observe that  $\mathbb T_r(a,b')\subseteq \mathbb T_r(a,b)$ and define
 \begin{equation}\label {E15} \mathbb T_r(B):= \frac {\mathbb T_r(a,b)}{i^{a,b}_{a' b}(r) \mathbb (T_r(a',b)) + \mathbb T_r(a, b')}.\end{equation}

In view of formula (\ref{E2}) (in subsection 2.1)    
\begin{equation}\label {E16} \mathbb T_r(B)= \hat \omega (i_{a'}^a(r), i_a^{b'}(r), i_{b'}^b(r) )\end{equation} 
with $i_{a}'^a(r), i_{a}^{b'}(r), i_{b'}^b(r)$ the linear maps 
 in  the sequence $$\xymatrix{ H_r(X_{a'})\ar[r]^{i_{a}'^a} &H_r(\tilde X_a)\ar[r]^{i_{a}'^{b'}}&H_r(\tilde X_{b'})\ar[r]^{i_{b'}^b}& H_r(\tilde X_b)}.$$

For $a''< a' <a$ and $b>b' > b''$  consider  boxes above diagonal $B_1, B_2, B$ with $B=(a'',a]\times (b'',b]$
and either $B_1 = (a'',a']\times (b'',b], B_2= (a',a]\times (b'',b]$ or $B_1=(a'',a]\times (b'',b'], B_2= (a'',a]\times (b',b].$  
In both cases $B= B_1\sqcup B_2.$ As in \cite{BU1} or \cite {B} one has the following proposition.
 
\begin{proposition} (cf. \cite {BU1}) \label {P3.3}\
The boxes $B_1, B_2, B$ as above induce the linear maps $i_{B_1} ^{B}(r)$ and $\pi_B^{B_2}(r)$ which make  the following  sequence exact.
$$\xymatrix { 0\ar[r]&\mathbb T_r(B_1) \ar[r]^{i_{B_1} ^{B}(r)} &\mathbb T_r(B) \ar[r]^{\pi_{B} ^{B_2}(r)}&\mathbb T_r(B_2)\ar[r] &0} $$
\end{proposition} 

\begin{obs} \label {O3.4}
As a consequence,  if $B_1\subset B$ is 
located in the down-left  corner of $B$ then the induced linear map 
$i_{B_1}^{B}(r): \mathbb T_r(B_1) \to \mathbb T_r(B)$  is injective and, if $B_2\subset  B$ is located in  the upper-right
corner of $B$ then the induced linear map $\pi_{B}^{B_2}(r): \mathbb T_r(B) \to \mathbb T_r(B_2)$ is surjective. 
\end{obs}
In  Figure 2 below  \footnote { In Figure 2 all boxes are supposed to be above diagonal  and not necessary in the first quadrant}  $B_{11}= (a'',a']\times (b'', b']$ is located in the  down-left corner of $B= (a'', a]\times (b'',b]$  and $B_{22}=(a', a]\times (b',b]$ in the upper-right corner of $B$.
\vskip .2in
\hskip 1 in \begin{tikzpicture} [scale=1.2]
\draw [<->]  (0,4) -- (0,0) -- (5,0);
\node at (-.2,3.5) {$b$};
\node at (-0.2,1.5) {$b'$};
\node at (-0.2,1) {$b''$};
\node at (1,-0.2) {$a''$};
\node at (3,-0.2) {$a'$};
\node at (4,-0.2) {$a$};
\draw  [dashed, ultra thick] (1,1) -- (4,1);
\draw [line width=0.10cm] (4,1) -- (4,3.5);
\draw [dashed, ultra thick] (1,1) -- (1,3.5);
\draw [line width=0.10cm] (1,3.5) -- (4,3.5);
\draw [line width=0.10cm] (1,1.5) -- (4,1.5);
\draw [line width=0.10cm] (3,1) -- (3,3.5);
\node at (2,1.25) {$B_{11}$};
\node at (2,2.5) {$B_{12}$};
\node at (3.5,1.25) {$B_{21}$};
\node at (3.5,2.5) {$B_{22}$};
\node at(2.4, -1) 
{Figure 2}; 
\end{tikzpicture}

\vskip .2in
  Define $$\boxed{ \hat \gamma^f_r(a,b):=\varinjlim_{\epsilon, \epsilon' \to 0}\mathbb T^f_r((a-\epsilon,a]\times (b-\epsilon',b]), \quad \gamma^f_r(a,b)= \dim \hat \gamma^f_r(a,b)
}$$ 
and observe that 
\begin{equation} \label {E17}\hat \gamma^f_r(a,b)=  \varinjlim_{\epsilon, \epsilon' \to 0}\frac {\mathbb T^f_r(a,b)}{i^{a,b}_{a-\epsilon,b} (\mathbb T^f_r(a-\epsilon,b)) + \mathbb T^f_r(a, b-\epsilon')}
= \frac {\mathbb T^f_r(a,b)}{ i_{<a,b} ^{a,b}  \mathbb T^f_r(<a.b) + \mathbb T^f_r(a, <b)} \end{equation}  

In view of (\ref{E15}) there is the obvious surjective linear map $$\boxed {\pi^\delta_{a,b}(r) :\mathbb T^f_r(a,b)\to \hat \gamma^f_r(a,b)} .$$ 
\vskip .2in

Suppose that  $f:\tilde X\to \mathbb R$ is a lift of a  tame  TC1-form $\omega.$ Suppose $H_r(\cdots)$ is a homology theory s.t. $\mathcal H^f_r(\tilde X_a, \tilde X_{<a})$ and $\mathcal H^f_r(\tilde X^a, \tilde X^{>a})$
are finite dimensional for any $a\in \mathbb R,$ hypotheses satisfied for both standard and Borel-Moore homologies.

As in section 5 of \cite{BU1}, where only  the standard homology is considered, one has the following result, valid for both standard and Borel-Moore homology theory.  

\begin{proposition}  \label {P3.5}\ 

Under the above hypotheses 
the supports of $\delta^f_r$ and $ \gamma^f_r$ are subsets of $CR(f)\times CR(f)$ with the following properties:
\begin{enumerate} 
\item If $(a,b)\in \supp \ \delta^f_r$ resp.  $(a,b)\in \supp \ \gamma^f_r$ then for any $g\in \Gamma$ one has   $(a+g, b+g) \in \supp\ \delta^f_r$ resp. $(a+g, b+g) \in \supp\ \gamma^f_r,$ hence  $\supp \delta^f_r$ and $\supp \gamma^f_r$ are $\Gamma-$invariant w.r. to the action $|mu(g, (x,y))= (g+x, g+y).$
\item For any $a\in \mathbb R$ $\supp \ \delta^f_r\cap  \mathbb R\times a$,  $\supp\ \delta^f_r\cap a\times \mathbb R,$  $\supp\ \gamma^f_r\cap \mathbb R\times a,$ $\supp\ \gamma^f_r\cap a\times \mathbb R$ are finite sets,  empty if $a\in \mathbb R \setminus CR(f).$
\item There exists a finite set of lines in the plane $\mathbb R^2,$ $\Delta^{\delta}_{t_i} (r)$  resp. $\Delta^\gamma_{t_i} (r),$   given by the equations $y= x+t^{\delta}_i,$ $i=1,2,\cdots N^\delta_r$ resp. 
$y=x+t^\gamma_i,$
$i=1,2,\cdots N^\gamma_r$  s.t. $\supp\ \delta^f_r \subset \cup_{i=1, \cdots, N^\delta_r} \Delta^\delta_{t^\delta_i}$ resp. 
$\supp\ \gamma^f_r \subset \cup_{i=1, \cdots, N^\gamma_r} \Delta^\gamma_{t^\delta_i}.$
\end{enumerate}
\end{proposition}

\begin{proof} (sketch)

Item (1) follows from the $\Gamma-$equivariance of the lift $f$ and the definitions, cf (\ref {E14}) and (\ref{E15}).   

To verify item (2) 
proceed as in  \cite{BU1}. 
 
Introduce:
\begin{itemize}
\item  for $a' <a,$ $b<b' $ 
\begin{equation*}
\begin{aligned} 
\mathbb F^f_r((a',a]\times b):=& \varinjlim _{\epsilon\to 0} \mathbb F^f_r((a',a]\times [b, b+\epsilon))\\ 
\mathbb F^f_r(a\times [b,b')):=& \varinjlim _{\epsilon\to 0} \mathbb F^f_r((a-\epsilon,a]\times [b, b'))
\end{aligned}
\end{equation*}
and note  that $\mathbb F^f_r(a\times [b,b'))= \mathbb F^{-f}_r((-b',-b]\times -a)$
\footnote {in view of the fact that ${\tilde X}^f_a= {\tilde X}^{-a}_{-f}$ and 
${\tilde X}^f_{<a}= {\tilde X}^{>-a}_{-f}$}, 
\item for $ a' <a \leq b' <b\leq \infty $
\begin{equation*}
\begin{aligned} 
\mathbb T^f_r((a',a]\times b):= &\varinjlim _{\epsilon\to 0} \mathbb T^f_r((a',a]\times (b-\epsilon, b]), \ \ 0< \epsilon <b-a \\
\mathbb T^f_r((a',b)\times b):= &\varinjlim _{b>a\to b} \mathbb T^f_r((a',a]\times b))\\ 
\mathbb T^f_r(a\times (b',b]):= &\varinjlim _{\epsilon\to 0} \mathbb T^f_r((a-\epsilon, a]\times (b', b])).
\end{aligned}
\end{equation*}
\end{itemize}
Clearly 
\begin{equation}\label {E18}
\begin{aligned}
\hat \delta^f_r(a,b)= &\varinjlim _{ \epsilon, \epsilon'\to 0}
 \mathbb F^f_r((a-\epsilon,a] \times [b, b+\epsilon'))= \varinjlim _{\epsilon\to 0} \mathbb F^f_r((a-\epsilon,a]\times b)=\varinjlim _{\epsilon' \to 0} \mathbb F^f_r (a\times [b, b+\epsilon')),\\ 
\hat \gamma^f_r(a,b)=& \varinjlim _{ \epsilon, \epsilon'\to 0} \mathbb T^f_r((a-\epsilon,a] \times (b-\epsilon', b])= \varinjlim _{\epsilon\to 0} \mathbb T^f_r((a-\epsilon,a]\times b)=\varinjlim _{\epsilon' \to 0} \mathbb T^f_r(a\times(b-\epsilon', b]),\\
 \hat \delta^f_r(a,b)=& \ \hat \delta^{-f}_r(-b, -a).
\end{aligned}
\end{equation}
 Observe that because $\mathcal H_r(\tilde X_a, \tilde X_{<a})$ 
 is a finite dimensional vector space,  the exact sequence (\ref{E11}) implies 
  $\dim \mathbb T_{r-1}(<a,a) <\infty $  and  
  $\dim \mathbb C_r(<a,a)<\infty$ for any  lift $f$ of the tame $\omega.$  
 
 In view of Proposition \ref{P3.3} and the exact sequence of the triple $(\tilde X_a\subseteq \tilde X_{b-\epsilon} \subseteq \tilde X_b)$
on has  
\begin{equation}\label {E19}
\dim \mathbb T_{r}((a',a]\times b) \leq \dim \mathbb T_{r}((a',b)\times b) \leq \dim \mathbb T^f_r(<b,b) \leq \dim \mathcal H_{r+1} (\tilde X_b, \tilde X_{<b})<\infty 
 \end{equation} 
 and in view of definitions
 \begin{equation} \label{E20}
\dim (\mathbb T^f_r(a,b) / \mathbb T^f_r(<a,b)) \leq \dim \mathbb C_r(<a,a) \leq \dim \mathcal H_{r} (\tilde X_a, \tilde X_{<a})<\infty 
\end{equation} 
\begin{equation} \label {E21}
\dim (\mathbb I^f_a(r) / \mathbb I^f_{<a} (r))  \leq \dim \mathbb C_r(<a,a) \leq \dim \mathcal H_{r} (\tilde X_a, \tilde X_{<a})<\infty    
 \end{equation} 
 
Note that $f$ a lift of the tame TC1-form $\omega$ makes $-f$ a lift of the tame TC-1 form $-\omega$ and  one has 
$$\mathbb I^b_f(r)= \mathbb I^{-f}_{-b}(r),\ \  
\mathbb I^{>b}_f(r)= \mathbb I^{-f}_{< -b}(r),$$ hence 
\begin{equation} \label {E22}
\dim (\mathbb I^{b}_f(r) / \mathbb I^{>b}_f(r))= \dim (\mathbb I_{-b}^{-f} (r) / \mathbb I_{-b}^{-f}(r)) < \dim \mathbb C_r^{-f} (<{-b}. {-b}) <\dim \mathcal H_r(\tilde X^b_f, \tilde X^{>b}_f) < \infty.
\end{equation}  

One can extend $\mathbb F_r(a\times [\alpha, \beta)),$  and $\mathbb F_r((\alpha, \beta]\times c)$ with $\alpha < \beta$ to 
$\alpha=-\infty$ or $\beta=\infty$ by defining  
\begin{equation}\label{E23}
\begin{aligned} 
\mathbb F_r(a\times [\alpha, \infty))  = &\varprojlim_{\alpha <\beta\to \infty } \mathbb F_r(a\times [\alpha,\beta))\\
\mathbb F_r(a\times (-\infty, \beta))  = &\varinjlim_{\beta>\alpha\to -\infty } \mathbb F_r(a\times [\alpha,\beta))\\
\mathbb F_r(a\times (-\infty, \infty))  = &\varinjlim_{\alpha \to -\infty } \mathbb F_r(a\times [\alpha, \infty))=\\
\mathbb F_r((-\infty, \beta] \times c)  = &\varprojlim_{\beta>\alpha\to -\infty } \mathbb F_r((\alpha, \beta] \times c)\\
\mathbb F_r((\alpha, \infty )\times c)  =& \varinjlim_{\alpha <\beta\to \infty } \mathbb F_r((\alpha, \beta] \times c)\\
\mathbb F_r((-\infty, \infty )\times c)  = &\varinjlim_{\beta\to \infty } \mathbb F_r((-\infty, \beta] \times c)= \\
\end{aligned}
\end{equation}

Similarly one can extend $\mathbb T_r (a\times (\alpha, \beta]), a \leq \alpha <\beta  <\infty $ to 
$\beta=  \infty$ 
\begin{equation}\label{E24}
\begin{aligned} 
\mathbb T_r(a\times (\alpha, \infty))  = \varinjlim_{\alpha <\beta\to \infty } \mathbb T_r(a\times (\alpha,\beta])\\
\end{aligned}
\end{equation}
and $\mathbb T_r ((\alpha, \beta]\times c), -\infty <\alpha <\beta <c$ to the case 
$\alpha=-\infty$ or / and $\beta=c$ 
by defining 
\begin{equation}\label{E25}
\begin{aligned} 
\mathbb T_r((-\infty, \beta] \times c)  = \varprojlim_{\beta>\alpha\to -\infty } \mathbb T_r((\alpha, \beta] \times c)\\
\mathbb T_r((\alpha, c) \times c)  = \varinjlim_{\alpha<\beta \to c } \mathbb T_r((\alpha, \beta] \times c)\\
\mathbb T_r((-\infty, c) \times c)  = \varinjlim_{\beta\to c } \mathbb T_r((-\infty, \beta] \times c)\\
\end{aligned}
\end{equation}

In view of (\ref {E21}) resp. (\ref{E19}), resp. (\ref{E20}), resp.( \ref {E22})  and of Propositions \ref{P3.1} and \ref{P3.3} the assignments 
\begin{enumerate}
\item  \ \  $ (- \infty, b) \ni t\rightsquigarrow \dim \mathbb F_r(a\times [t,b)),  -\infty <t \leq b$
\item \ \   $(a',b)\ni t\rightsquigarrow \dim \mathbb T_r((a',t]\times b),  -\infty \leq a' <t <b$
\item \ \    $(b',\infty )\ni t\rightsquigarrow \dim \mathbb T_r(a\times (b', t]),  a\leq b' <t <\infty$ .
\item  \ \   $(a',\infty)\ni t\rightsquigarrow \dim \mathbb F_r((a',t]\times b),  a'\leq t <\infty$
\end{enumerate}
are bounded $\mathbb Z_{\geq 0}-$valued functions  with (1) decreasing and (2), (3) and (4) increasing in $t,$ hence step functions with finitely many jumps  at 
\begin{enumerate}
\item $-\infty < t^a_1 <t^a_2 <\cdots < t^a_{N_a} <\infty$  
\item $-\infty < t^1_b <t^2_b <\cdots < t_b^{M_b} <b $
\item $a < t^a_1 <t^a_2 <\cdots < t^a_{M_a} <\infty $
\item $-\infty < t^1_b <t^2_b <\cdots < t_4^{N_b} <\infty$ 
\end{enumerate}
$N_a, N_b, M_a, M_b$ nonnegative integers.

 The jump at $t$  
given, in view of (\ref{E18}), by  

$\lim_{\epsilon\to 0}\dim ( \mathbb F_r(a\times [t,b))/ \mathbb F_r(a\times [t+\epsilon,b))= \lim_{\epsilon\to 0} \dim \mathbb F_r(a\times [t,t+\epsilon)) = \delta_r(a,t)$ 
in view of Proposition \ref {P3.1}
 for assignment (1), 

$\lim_{\epsilon\to 0}\dim ( \mathbb T_r((a',t]\times b)/ \mathbb T_r((a', t-\epsilon] \times ,b))= \lim_{\epsilon\to 0} \dim \mathbb T_r((t-\epsilon, t] \times b) = \gamma_r(t,b)$ 
in view of Proposition \ref {P3.3}
 for assignment (2), 

$\lim_{\epsilon \to 0}  \dim ( \mathbb T_r(a \times (b', t]) / \mathbb T_r(a \times (b', t-\epsilon])=\lim_{\epsilon\to 0} \dim \mathbb T_r(a\times (t-\epsilon,t]) = \gamma_r(a,t)$ 
in view of Proposition \ref {P3.3}
 for assignment (3), 

$\lim_{\epsilon\to 0}\dim ( \mathbb F_r((a',t]\times b)/ \mathbb F_r(a', t-\epsilon]\times b))= \lim_{\epsilon\to 0} \dim \mathbb F_r((t-\epsilon,t]\times b) = \delta_r(t,b)$ 
in view of Proposition \ref {P3.1}
 for assignment (4), 

In particular we have the following:

\begin{enumerate} 
\item  Suppose $\dim H_r( X_a, X_{<a}) <\infty.$  Then 
\begin{enumerate} 
\label{E26}
\item $\hat \delta^f_r(a,t)=0$ of $t\ne \{ t^a_1, \cdots t^a_{N_a}\},$ hence  the map $\delta^f_r: a\times \mathbb R\to \mathbb Z_{\geq 0}$ is a configuration of points,
\item for any $x,y$ with $-\infty \leq x\leq y\leq \infty$   
\begin{equation}
\mathbb F_r(a\times [x, y))\simeq \bigoplus_{ x\leq t <y}\hat \delta^f_r(a, t).\end{equation}
\end{enumerate}

\item Suppose $\dim H_r( X_b, X_{<b}) <\infty.$  Then 
\begin{enumerate} 
\label{E27}
\item $\hat \delta^f_r(t, b)=0$ of $t\ne \{t_b^1, \cdots t_b^{M_b}\}$  hence  the map $ \gamma^f_r:  (-\infty, b)\to \mathbb Z_{\geq 0}$ is a configuration of points.
\item for any $x, y$ with   $-\infty \leq x <y<b$ one has 
\begin{equation}
\begin{aligned}\mathbb T_r((x,y]\times b)\simeq \bigoplus_{x <t \leq y} {\hat \gamma^f_r(t,b)} \\ 
\mathbb T_r((x,b)\times b)\simeq \bigoplus_{x <t <b} {\hat \gamma^f_r(t,b)} .
\end{aligned} 
\end{equation}
\end{enumerate}

\item 
Suppose $\dim H_r( X_a, X_{<a}) <\infty.$   Then 
\begin{enumerate} 
\label{E28}
\item $\hat \gamma ^f_r(a,t)=0$ of $t\ne \{ t_b^1, \cdots t_b^{M_b}\}$ hence  the map $\gamma^f_r: a\times (a, \infty)\to \mathbb Z_{\geq 0}$ is a configuration of points,
\item  for any $x, y$with   $a\leq x <y\leq \infty$ one has  
 \begin{equation}
\mathbb T_r(a\times (x, y])\simeq \bigoplus_{ x< t \leq y}  {\hat \gamma^f_r(a, t)}\end{equation}
\end{enumerate}

\item Suppose $\dim H_r( X^b, X^{>b}) <\infty.$  Then 
\begin{enumerate} 
\label{E29}
\item $\hat \delta^f_r(t, b)=0$ of $t\ne \{t_b^1, t_b^2, \cdots t_b^{N_b}\}$ hence  the map $\delta^f_r:  \mathbb R\times b \to \mathbb Z_{\geq 0}$ is a configuration of points,
\item for any $-\infty\leq x\leq y \leq \infty$ one has 
\begin{equation} \mathbb F_r((x,y]\times b)\simeq \bigoplus_{ x<t \leq y}\hat \delta^f_r(t,b) \end{equation}
\end{enumerate}
\end{enumerate}
Part (a) of (1), (2), (3) and (4) establish item 2.  Part (b)  provide calculations used below. 
\vskip .1in
To prove  Item 3  
first observe that both $\supp  \delta_r^f$ and $\supp  \gamma^f_r$ are  $\Gamma-$invariant w.r. to the  diagonal action  of $\Gamma$ on $CR(f)\times CR(f)\subset \mathbb R\times \mathbb R,$
$\mu (g, (x,y))\rightsquigarrow (g+x, g+y)).$ 

A collection $\mathcal B$ of elements of $\supp \delta^f_r$ or $\supp \gamma^f_r$ is called a {\it base} if :
\begin{enumerate}
\item for any $v\in \supp \delta^f_r$ resp. $v \in \supp \gamma^f_r$  there exists $u\in \mathcal B$ and $g\in \Gamma$ s.t. $v= \mu(g, u)$ 
\item If $u_1, u_2 \in \mathcal B$ and $u_2= \mu(g, u_1)$ then $g=0.$
\end{enumerate}

For a base $\mathcal B$  define 
$$^{\mathcal B}\delta^\omega (t) := \sum _{(a,b) \in \mathcal B \mid b-a=t} \underline \delta^f_r(a,b), \ \ \underline \delta(a,b)= b-a$$ resp.
$$^{\mathcal B}\gamma^\omega(t) := \sum _{(a,b) \in \mathcal B \mid b-a=t} \underline \gamma^f_r(a,b), \ \ \underline  \gamma(a,b)= b-a$$

which, when $\mathcal B$ is finite, are clearly a $\mathbb Z_{\geq 0}-$valued configurations on $\mathbb R$ resp. $\mathbb R_{>0}$
and in view of the properties of the base are independent of $\mathcal B.$ 

Each finite base $\mathcal B$ provides a finite collection of real numbers $t= b-a, (a,b)\in \mathcal B,$ which in view of the properties of the base is independent of $\mathcal B$ and of the lift $f$ 
hence define the configurations $\delta^\omega_r$ and $\gamma^\omega.$

To finalize the proof one needs  to check the existence of bases$\mathcal B.$ In the case $\delta$ we proceed as follows.
For any $o\in CR(f)/ \Gamma$ one chooses $a^o\in CR(f)$ $a^o\in o ;$ then  the collection 
$$\mathbb B:= \bigcup_{o\in CR(f)/\Gamma}   \{ (a^o, t^{a^o}_1)), \cdots (a^o, t^{a^o}_{N_{a^o}}) \}$$
is a base. 
One can also   one chooses $b_o\in CR(f)$ $b_o\in o;$ then  the collection 
$$\mathbb B:= \bigcup_{o\in CR(f)/\Gamma}   \{ (b_o, t_{b_o}^1)), \cdots (b_o, t_{b_o}^{N_{b_o}}) \}$$
is a base. 
The case of $\gamma$  is entirely similar.

\end{proof} 
\vskip .2in

Define    $$\mathbb T_r(-\infty,a) :=  \varprojlim _{a> \alpha\to -\infty} \mathbb T_r(\alpha, a)$$

$$\mathbb T_r(<a, \infty):= \varinjlim _{a< x\to \infty} T_r(<a,x)$$  

$$\mathbb T_r(a, \infty):= \varinjlim _{a< x\to \infty}  T_r(a,x)$$  

and denote by

$$
\hat \lambda_r(a) := 
\img (\mathbb T_r(-\infty, a) \to T_r( <a,a))$$ which is a f.d. vector space when $\mathcal H_r(X_a, X_{<a})$ is finite dimensional.

\begin {obs} \label {O3.6}\  

\begin{enumerate}
\item  $\mathbb T_r (a\times (a,\infty))= \mathbb T_r (a,\infty)/ i \mathbb T_r( <a,\infty)$  
\item $\mathbb T_r((-\infty,a)\times a) = \mathbb T_r (<a.a)/ i \mathbb T_r(-\infty,a)= \mathbb T_r (<a.a)/ i \mathbb T_r(-\infty,a)= \mathbb T_r (<a.a)/ (\hat \lambda_r(a))$
\end{enumerate}
\end{obs}
\begin{proof}
Apply the definitions.

\end{proof}

\begin{theorem} \label {T3}\

Suppose $\dim H_r(X_a, X_{<a}) <\infty.$  Then:

\begin{enumerate}
\item $\mathbb T_r(<a,a)= \hat \lambda_r(a) \oplus \mathbb T_r((-\infty, a)\times a)=  \hat \lambda_r(a) \oplus \bigoplus _{-\infty<x <a} { \hat \gamma^f_{r} (x,a)}$
\item $  \varinjlim_{\epsilon\to 0}\coker (H_r(X_{a-\epsilon})\to H_r(X_a))= \mathbb I_a(r) / \mathbb I_{<a} (r) \oplus  \bigoplus _{x<a}  { \hat \gamma_r (a,x)}$
\item $\mathcal H_r( X_a, X_{<a})= \mathbb I_a(r) / \mathbb I_{<a} (r) \oplus  \bigoplus _{x>a}  { \hat \gamma_r (a, x)}\oplus  \hat \lambda_{r-1}(a) \oplus \bigoplus _{x<a} { \hat \gamma^f_{r-1} (x,a)}$
\end{enumerate}
\end{theorem}
\begin{proof} 
Item 1 follows from Observation \ref {O3.6} item 2 and (\ref {E27}).

\vskip .1 in
 To check items 2 ad 3 some  additional notations,  not existing in \cite{BU1}, are be of help.
\begin{equation} \label{E30}
\begin{aligned}
\mathbb T_r(a', a:t):= &\ker(i_{a',t}^{a,t}), a' < a < t\leq \infty,\\
\mathbb C_r(a', a; t):= &\coker (i_{a',t}^{a,t}),, a' < a < t\leq \infty,\\
\mathbb I_r(a',a; t): = &( \mathbb I^f_a(r)\cap \mathbb I^t_f(r)) / ( \mathbb I^f_{a'}(r)\cap \mathbb I^t_f(r)) \subseteq \mathbb I^f_a(r)/ ( \mathbb I^f_{a'}(r) \ , a' <a, \\ 
\end{aligned}
\end{equation}
where $i_{a',t}^{a,t}) :\mathbb T_r(a',t)\to \mathbb T_r(a,t)$ is the obviously induced map considered in formula (\ref{E15}).

With these notations one has  the diagram (\ref{D31}) below whose  rows and columns are exact sequences
\begin{equation}\label {D31}
\xymatrix{     &0\ar[d]                                               &                                    &                                                           &\\
&\mathbb C_r(a-\epsilon, a;\infty)\ar[d]&&
 &\\
0\ar[r]&\mathbb C_r(a-\epsilon, a)\ar[d]\ar[r]& H_r(\tilde X_a, \tilde X_{a-\epsilon}) \ar[r]& \mathbb T_{r-1}(a-\epsilon, a) \ar[r]&0\\
& \mathbb I_r(a-\epsilon, a)\ar[d] &&&\\
&0&&&}\end{equation} 
Note that we have the following   filtrations:
\begin{enumerate}
\item $\{ \mathbb I_r(a-\epsilon,a) \cdots \supseteq  \mathbb I_r(a-\epsilon, a; t) \supseteq \mathbb I_r(a-\epsilon, a; t+\epsilon') \supseteq \cdots\}$ indexed by $t\in \mathbb R$
with  the property
 
 $\boxed{ \mathbb I_r(a-\epsilon, a; t) /  \mathbb I_r(a-\epsilon, a; t+\epsilon)= \mathbb F_r((a-\epsilon, a]\times [t, t+\epsilon))},$  

\item $\{\mathbb C_r(a-\epsilon, a)\cdots \supseteq \mathbb C_r(a-\epsilon, a; t) \supseteq \mathbb C_r(a-\epsilon, a; t-\epsilon) \supseteq \cdots\}, \ t-\epsilon >a$ indexed by $t\in (a,\infty),$
which  in view of (\ref{E11}) satisfies
 
 $\boxed{\mathbb C_r(a-\epsilon, a; t) /  \mathbb C_r(a-\epsilon, a; t-\epsilon')= \mathbb T_r(a-\epsilon, a]\times (t- \epsilon, t])
 }.$ 
\end{enumerate}

Recall / observe that: 
\begin{enumerate}[label = (\roman*)]
\item $\mathcal H_r(\tilde X_a, \tilde X_{<a}) =\varinjlim_{\epsilon\to 0} H_r(\tilde X_a, \tilde X_{a-\epsilon}),$ \quad 

$ \mathbb T_r(<a,a):=\varinjlim_{\epsilon\to 0} \mathbb T_r(a-\epsilon, a),$ 

$ \mathbb C_r(<a,a):=\varinjlim_{\epsilon\to 0} \mathbb C_r(a-\epsilon, a),$ 

$\mathbb I_r(<a,a):=\varinjlim_{\epsilon\to 0} \mathbb I_r(a-\epsilon, a)= \mathbb I_r(a)(r) /\mathbb I_{<a}(r).$

\item  $\mathbb I_r( <a,a; t):=\varinjlim_{\epsilon\to 0} \mathbb I_r( a-\epsilon,a; t),$ and   

$\hat \delta ^f_r(a, t):=\varinjlim_{\epsilon \to 0} \mathbb I_r( <a, a; t)/  \mathbb I_r( <a, a; t+\epsilon)$
\item $\mathbb C_r(<a,a; t):=\varinjlim_{\epsilon\to 0} \mathbb C_r(a-\epsilon, a; t), a <t$  and 

$\hat \gamma^f_r(a, t):=\varinjlim_{\epsilon \to 0} \mathbb C_r(<a,a; t)/\mathbb C_r(<a,a; t-\epsilon ), a<t-\epsilon .$
\item $\mathbb T_r(<a,a; t):=\varinjlim_{\epsilon\to 0} \mathbb T_r(a-\epsilon, a; t), a<t .$ 
\end{enumerate}
By passing to limit  when $\epsilon \to 0$ the diagram  (\ref{D31}) becomes  

\begin{equation}\label {D32}
\xymatrix{     &0\ar[d]                                               &                                    &                                                           &\\                                 
&\mathbb C_r(<a, a;\infty)\ar[d]&&&\\
0\ar[r]&\mathbb C_r(<a, a)\ar[d]\ar[r]& \mathcal H_r(\tilde X_{a}, \tilde X_{<a}) \ar[r]& \mathbb T_r(<a, a)\ar[r]&0 .\\
&\mathbb I_r(<a,a)\ar[d] &&&\\
&0&&& .} . \end{equation} 

The exactness of the columns and rows in the diagram 

$\xymatrix { &&0\ar[d]&0\ar[d]&&\\
0\ar[r]&\mathbb T_r (a', a;t)\ar[d]\ar[r]& \mathbb T_r(a',t)\ar[r]^{i_{a',t}^{a,t}}\ar[d] &\mathbb T_r(a,t) \ar[r]\ar[d] &\mathbb C_r(a',a;t)\ar[d]\ar[r]&0\\
0\ar[r]&\mathbb T_r (a', a)\ar[r]& \mathbb H_r(\tilde X_{a'})\ar[d]\ar[r]^{i_{a'}^a} &\mathbb H_r(\tilde X_a)\ar[d] \ar[r] &\mathbb C_r(a',a)\ar[r]&0\\
&&H_r(\tilde X_t)\ar[r]^{=}&H_r(\tilde X_t)&&}$
 \vskip .1in

 for $a'<a <t$ and the commutativity of the diagram  below for for $a'<a <t <t'$

$\xymatrix { 
0\ar[r]&\mathbb T_r (a', a;t)\ar[d]\ar[r]& \mathbb T_r(a',t)\ar[r]^{i_{a',t}^{a,t}}\ar[d] &\mathbb T_r(a,t) \ar[r]\ar[d] &\mathbb C_r(a',a;t)\ar[d]\ar[r]&0\\
0\ar[r]&\mathbb T_r (a', a;t')\ar[d]\ar[r]& \mathbb T_r(a',t')\ar[r]^{i_{a',t}^{a,t}}\ar[d] &\mathbb T_r(a,t') \ar[r]\ar[d] &\mathbb C_r(a',a;t')\ar[d]\ar[r]&0\\
0\ar[r]&\mathbb T_r (a', a)\ar[r]& \mathbb H_r(\tilde X_{a'})\ar[r]^{i_{a'}^a} &\mathbb H_r(\tilde X_a) \ar[r] &\mathbb C_r(a',a)\ar[r]&0
}$

 imply the injectivity of  all arrows  $\mathbb T_r (a', a;t)\to \mathbb T_r (a', a;t')\to \mathbb T_r (a', a)$ 
 and 
 
 $\mathbb C_r (a', a;t)\to \mathbb C_r (a', a;t')\to \mathbb C_r (a', a).$ 

As already noticed, in view of the exactness of the row in (\ref{D32}), if $\dim \mathcal H_r(\tilde X_a, \tilde X_{<a})$ is finite then so are $\dim \mathbb T_r( <a,a, \infty), \dim \mathbb C_r(<a, a;\infty), \dim \mathbb I_a(r)/ \mathbb I_{<a}(r)$ and then so are 
$\dim \mathbb T_r( <a, a; t), \dim \mathbb C_r(<a, a; t)$ and $\dim \mathbb I_r(<a,a).$ Then the functions $\dim \mathbb T_r(<a, a; t)$ and  $\dim \mathbb C_r(<a,a; t)$ resp. $\dim \mathbb I_r(<a,a; t)$ are integer valued with finitely many jumps located at $t$ where $\gamma^f_r(a,t)\ne 0$ resp. $\delta^f_r(t,a)\ne 0 .$ They are also continuous from the right.
In view of (ii), (iii), (iv) above one has
\begin{enumerate} [label = (\alph*)]
\item $\mathbb I_r (<a,a)= \oplus  _{t< a}  \hat \delta_r(t,a),$
\item $\mathbb C_r (<a,a;\infty)= \oplus _{t>a}  \hat \gamma_r(t,a),$
\end{enumerate}
and combined with item 1 one has 
 $$\boxed {\dim\mathcal H_r(\tilde X_a, \tilde X_{<a})=\sum_{t<a} \delta^f_r(t,a) 
 +\sum _{t>a}\gamma^f_r(a,t) + \sum_{t<a}\gamma^f_{r-1}(t,a) +\lambda^f_{r-1}(a)},$$
which finalizes the proof of Theorem  \ref{T3}.

In case $H_r$ is the standard homology this is   Corollary 4 in \cite {BU1}.

\end{proof}

\subsection {Splittings}
\vskip .1in
 Recall from  \cite{BU1} section 4:

A splitting  
$i^{\delta}_{a,b}(r): \hat\delta^f_r(a,b)\to \mathbb F_r(a,b),$ for $(a,b)\in \supp \delta^f_r \subset CR_r(f)\times C_r(f),$
is a right inverse of the canonical projection  $\pi_{a,b}^\delta(r): \mathbb F_r(a,b)\to \hat \delta^f_r(a,b)$ i.e. $ \pi_{a,b}^\delta (r)\cdot i_{a,b}^\delta(r)= id.$  

A splitting $i^{\gamma}_{a,b}(r): \hat\gamma^f_r(a,b)\to \mathbb T_r(a,b),$ for $(a,b)\in \supp\ \gamma^f_r,$ is a right inverse of the canonical projection  $\pi_{a,b}^\gamma (r): \mathbb T_r(a,b)\to \hat \gamma^f_r(a,b).$ 

For a box $B=(a',a]\times [b,b') ,\  a' < a, b < b'$, 
resp. for a box above diagonal $B= (a', a] \times (b', b],  \ a' < a \leq b' <b$ one writes $i^B_{a,b}(r)$ for the composition of $i_{a,b}^\delta(r)$ resp. $i_{a,b}^\delta(r)$ with the projection $\mathbb F_r(a,b)\to \mathbb F_r(B)$ resp $\mathbb T_r(a,b)\to \mathbb T_r(B)$. Clearly the linear maps $i_{a,b}^B$  remain splittings of the canonical projections 
$\mathbb F_r(B)\to \hat \delta^f_r(a,b),$ resp. $\mathbb T_r(B)\to \hat \gamma^f_r(a,b).$ 

One  extends these splittings to any $(\alpha, \beta)\in B$ as being  the linear injective maps  $i^B_{\alpha, \beta}(r): \hat\delta^f_r(a,b) \to \mathbb F_r(B)$ resp. 
$i^B_{\alpha, \beta}(r): \hat \gamma^f_r(a,b) \to \mathbb T_r(B)$  
defined by the composition $i_{B'}^B(r)\cdot i_{\alpha, \beta}^{B'}(r)$ where $B'= (a' \alpha] \times [\beta, b'), \ a' <\alpha \leq a, b \leq\beta <b'$ resp. $B'= (a', \alpha] \times (b',\beta], \   a' <\alpha \leq a, \ b'< \beta \leq b.$
\vskip .1in

For $f:\tilde X\to \mathbb R$ a lift of a tame TC1-form $\omega$,  a collection $S^\delta$ resp. $S^\gamma$ of splittings $S^\delta= \{i^\delta _{a,b}(r)\}$ resp. $S^\gamma= \{i^\gamma _{a,b}(r)\}$  
is called {\it $\Gamma-$compatible}  if for any $(a,b)\in \supp \delta^f_r$ resp. $(a,b)\in \supp \gamma^f_r$ the diagrams 

\begin{equation}\label {D33}
\xymatrix 
{\mathbb F_r^f(a,b) \ar[r]^{\langle g\rangle} & \mathbb F_r(g+a, g+b) \\ 
 \tilde {\delta}^f_r(a,b)\ar[u]^{i^\delta_{a,b}(r)}\ar[r]^{\langle g\rangle} &\hat \delta_r^f(a+g, b+g)\ar[u]^{i^\delta_{a+g,b+g}(r)}}  \ \ resp. \ \ \xymatrix 
{\mathbb T_r^f(a,b) \ar[r]^{\langle g\rangle} & \mathbb T_r(g+a, g+b) \\ 
 \tilde {\gamma}^f_r(a,b)\ar[u]^{i^\gamma_{a,b}(r)}\ar[r]^{\langle g\rangle} 
  &\hat \gamma_r^f(a+g, b+g)\ar[u]^{i^\gamma_{a+g,b+g}(r)}} 
  \end{equation}
 remain commutative.
Note that collections of \ $\Gamma-$compatible splittings $S^\delta$ resp. $S^\gamma$ exist.

\noindent Indeed, if one  chooses one point $(a_i,b_i)$ in each $\Delta^\delta _{t_i}(r) \cap \supp \delta^f_r$ resp. $\Delta^\gamma _i(r) \cap \supp \gamma^f_r$ and a splitting $i^{\delta}_{a_i, b_i}(r)$ resp. $i^{\gamma}_{a_i, b_i}(r)$ and one defines $$i^{\cdots}_{a_i+g, b_i+g}(r)
= \langle g\rangle \cdot  i^{\cdots}_{a_i, b_i}(r)\cdot  \langle g\rangle^{-1}$$   then,  since  $(a_i +g, b_i+g)$ exhaust all points in the support of $\hat{\delta}^f_r$ resp. $\hat{\delta}^\gamma_r,$
one obtains  a collection of $\Gamma-$compatible splittings $S^\delta$ resp. $S^\gamma.$ 

Given a collection of splittings $S^\delta$ 
and a set $A\subset \supp\  \delta^f_r,$ as  in \cite{BU1} Proposition 4, one shows that the linear map 
$$^{S^\delta}  I_A(r)= (\oplus_{(\alpha, \beta)\in A} \  i^\delta _{\alpha, \beta}(r)) : \oplus \hat \delta^f_r(\alpha, \beta) \to H_r(\tilde X)$$  is injective, has the image contained in $\mathbb F_r(a,b)$ provided  $A \subset (-\infty, a]\times [b,\infty),$  and remains injective when composed by $\pi_{a,b}^B: \mathbb F_r(a,b)  \to \mathbb F_r(B)$ resp. $\pi_{a,b}^I: \mathbb F_r(a,b)  \to \mathbb F_r(I)$
where $B= (a',a]\times [b, b')$ resp. $I= a\times [b,b')$ or $I= a\times [b,b'),$  $a'<a, b<b'.$
Moreover, if one write $A= A(I):=\supp \delta^f_r \cap I$  the composition 
$$^{S^\delta}  I_{A(I)}(r)= \oplus_{(\alpha, \beta)\in A(I)} \  i^\delta _{\alpha, \beta}(r) : \oplus \hat \delta^f_r(\alpha, \beta) \to \mathbb F_r(I)$$ is an isomorphism.
 
 The same remains true for a collection of splittings $S^\gamma,$  $A \subset \{\alpha, \beta \in \supp \gamma^f_r \mid \alpha\leq a\}$ and the map 
$$^{S^\gamma}  I_A(r)= (\oplus_{(\alpha, \beta)\in A} \  i^\gamma _{\alpha, \beta} (r)) : \oplus \hat \gamma^f_r(\alpha, \beta) \to \mathbb T_r(a, \infty).$$ and 
$$^{S^\gamma}  I_{A(B)}(r): \oplus_{ (\alpha, \beta)\in A(B)}  \hat \gamma^f_r(\alpha, \beta) \to \mathbb T_r(B)$$ resp. $$^{S^\gamma}  I_{A(I)}(r): \oplus_{ (\alpha, \beta)\in A(I)}  \hat \gamma^f_r(\alpha, \beta) \to \mathbb T_r(B)$$  for $B= (a',a] \times (b',b]$ resp. $ I= (a',a]\times b$ or $a\times (b',b],$  $a' <a \leq b' <b.
\   $Precisely $^{S^\gamma}  I_A(r)$ is injective and $^{S^\gamma}  I_{A(I)} (r)$ with target $\mathbb T_r(I)$ is isomorphism. 

\section {Proof of Poincar\'e duality, Theorem \ref {TP} }

Suppose $a,b$ regular values for $f: \tilde M\to \mathbb R$  a lift of a locally smooth \footnote { in the neighborhood of any point there exists coordinates in which  $\omega$ is a differential closed one form whose components are given by polynomial maps} hence tame,TC1-form $\omega$ on a topological closed manifold $M,$
hence the levels $\tilde M(a), \tilde M(b)$  are codimension one submanifolds of $\widetilde M.$ 
Poincar\'e duality for the topological manifolds with boundary $\tilde M_a$ and  
$\tilde M^b$  combined with the excision property for  standard cohomology 
provide the canonical isomorphisms  " $PD_r$ " from the Borel-Moore homology $^{BM} H_r$ to the standard  cohomology  $H^{n-r}$ and then 
the  commutative diagrams below.

\begin{equation}\label{D34}
\scriptsize
\xymatrix{&H^{\rm BM}_r( \widetilde M_a)\ar[d]^{{\rm PD}^1_a}\ar[r]^{i_a(r)}  \quad &H^{\rm BM}_r( \widetilde M)\ar[d]^{\rm PD}\ar[r]^{j_a(r)} &H^{\rm BM}_r( \widetilde M, \widetilde M_a)\ar[d]^{{\rm PD}^2_a}&\\
&H^{n-r}( \widetilde M, \widetilde M^a)\ar[d] \ar[r]^{s^a(n-r)} \quad&H^{n-r}( \widetilde M)\ar[d] \ar[r]^{r^a(n-r))} &H^{n-r}( \widetilde M^a)\ar[d]&\\
&(H_{n-r}( \widetilde M, \widetilde M^a))^\ast \ar[r]^-{(j^a(n-r))^\ast} \quad&(H_{n-r}( \widetilde M))^\ast \ar[r]^{(i^a(n-r))^\ast} &(H_{n-r}( \widetilde M^a))^\ast&}
\end{equation}
\begin{equation}\label {D35}
\scriptsize
\xymatrix
{&H^{\rm BM}_r( \widetilde M^b)\ar[d]^{{\rm PD}_1^b} \ar[r]^{i^b(r)} \quad &H^{\rm BM}_r( \widetilde M)\ar[d]^{\rm PD}\ar[r]^{j^{b}(r)} &H^{\rm BM}_r( \widetilde M, \widetilde M^b)\ar[d]^{{\rm PD}_2^b}&\\
&H^{n-r}( \widetilde M, \widetilde M_b)\ar[d] \ar[r]^{s_b(n-r)}  \quad &H^{n-r}( \widetilde M)) \ar[d] \ar[r]^{r_b(n-r)} &H^{n-r}( \widetilde M_b)\ar[d]\\
&(H_{n-r}( \widetilde M, \widetilde M_b))^\ast \ar[r]^{(j_b(n-r))^\ast} \quad &(H_{n-r}( \widetilde M))^\ast \ \ar[r]^-{(i_b(n-r))^\ast} &(H_{n-r}( \widetilde M_b))^\ast &.}
\end{equation}

In view of (\ref {D34}) and  (\ref {D35})one has 
$$ \begin{aligned}  ^{BM} \mathbb F_r(a,b) :=  \img i_a(r) \cap \img i^b(r)= \ker j_a(r) \cap \ker j^b(r) \\  \ker (i^a(n-r))^\ast \cap \ker (i_b(n-r))^\ast 
 \simeq (\coker (i_b (n-r)\oplus i^a(n-r)))^\ast = :\mathbb G^f_{n-r} (b,a)^\ast. \end{aligned}$$
The first equality holds by exactness of the first row in these diagrams, the second by the equality  of the top- and bottom-right horizontal arrows, the third by the linear algebra duality and  the fourth by the definition of $\mathbb G_{n-r}.$

This in turn implies that for $a' <a, b <b'$ regular values with the boxes $B$ and $B'$ given by  $B= (a',a]\times [b,b'),$ $B'= (b, b']\times [a', a)$ the diagram (\ref {D36} is commutative with the horizontal  arrows all isomorphisms, in particular 
\begin{equation} \label {E36} 
PD_r(B) : ^{BM} \mathbb F^f_r((a' a]\times[b,b'))\to (\mathbb G^f_{n-r} ((b, b']\times[a', a))^\ast
\end{equation}
is an isomorphism.
\begin{equation} \label {D37}
\scriptsize
\xymatrix @C-0.5pc{^{\rm BM}\mathbb F^f_r(a',b')\ar[d]\ar[rr]^{{\rm PD}_r(a',b')}&&(\mathbb G^f_{n-r}(b',a'))^\ast\ar[d]&\\
^{\rm BM}\mathbb F^f_r(a,b)\ar[dd]\ar[dr]^{^{\rm BM} \pi^B_{ab,r}}\ar[rr]^{{\rm PD}_r(a,b)}&&(\mathbb G^{\widetilde f}_{n-r}(b,a))^\ast\ar[dr]^{(u_{n-r}(B'))^\ast}\ar[dd]^--{(p^{ba}(r))^\ast}&\\
&^{\rm BM}\mathbb F^f_r(B)\ar@/^1pc/[rr]^{{\rm PD}_r(B)} &&(\mathbb G^{\widetilde f}_{n-r}(B'))^\ast  \ar[d]|{^{(\theta_{n-r}(B'))^\ast}}&\\
H^{\rm BM}_r(\widetilde M)\ar[rr]^{{\rm PD}_r}&&(H_{n-r}(\widetilde M))^\ast &(\mathbb F^{\widetilde f}_{n-r}(B')^\ast}
\end{equation}
\vskip .2in

Suppose $a <b <c <d$ with all $a, b, c, d$ regular values and consider the box above diagonal $B= (a,b]\times (c,d]$. 
Poincar\'e duality and excision property provide the following commutative diagram 

\begin{equation}\label {D38}
\scriptsize 
\xymatrix@C-1.5pc{
(1)&^{BM} H_r(\tilde M^f_a)\ar[d]^{\rm PD} \ar[r]^\alpha & ^{BM}H_r(\tilde M^f_b)\ar[d]^{\rm PD}\ar[r]^\beta  &^{BM}H_r(\tilde M^f_c)\ar[d]^{\rm PD} \ar[r]^\gamma & ^{BM}H_r(\tilde M^f_d)\ar[d]^{\rm PD}\\ 
(2)&H^{n-r}(\tilde M,\tilde M^a_f) \ar[r]& H^{n-r}(\tilde M,\tilde M^b_f)\ar[r] &H^{n-r}\tilde M,\tilde M^c_f) \ar[r]& H_r^{n-r}(\tilde M,\tilde M^d_f)\\
(3)&H^{n-r-1}(\tilde M_f^a) \ar[r]\ar[u]^{\partial}& H^{n-r-1}(\tilde M_f^b)\ar[r]\ar[u]^{\partial} &H^{n-r-1}(\tilde M_f^c) \ar[r]\ar[u]^{\partial}& H^{n-r-1}(\tilde M_f^d)\ar[u]^{\partial}\\
(4)&(H_{n-r-1}(\tilde M_f^a))^\ast \ar[r]\ar[u]^{=}& (H_{n-r-1}\tilde M_f^b))^\ast\ar[r]\ar[u]^{=} &(H_{n-r-1}(\tilde M_f^c))^\ast \ar[r]\ar[u]^{=}& (H_{n-r-1}(\tilde M_f^d)^\ast\ar[u]^{=}\\
(5)&(H_{n-r-1}(\tilde M^{-f}_{-a}))^\ast \ar[r]^{{\gamma'}^\ast}\ar[u]& (H_{n-r-1}(\tilde M^{-f}_{-b}))^\ast\ar[r]^{{\beta'}^\ast}\ar[u] &(H_{n-r-1}(\tilde M^{-f}_{-c}))^\ast \ar[r]^{{\alpha'}^\ast}\ar[u]& (H_{n-r-1}(\tilde M^{-f}_{-d}))^\ast .\ar[u] }
\end{equation}  

In view of Theorem \ref {T2.3} item 2 this diagram  implies that $\omega (\alpha, \beta, \gamma)= \omega({\gamma'}^\ast,{\beta'}^\ast, {\alpha'}^\ast)$ (cf subsection 2.1 for definitions) 
which by Theorem \ref {T2.3} 
item 1 is canonically isomorphic to 
$\omega(\gamma', \beta', \alpha')^\ast.$ In particular one has the  isomorphism 
\begin{equation} \label {E39}
^{BM}PD_r(B) : ^{BM} \mathbb T^f_r (a,b]\times (c,d])\to  (\mathbb T^{-f}_{n-r-1} ((-d, -c]\times (-b,-a]))^\ast.
\end{equation}
\vskip .2in

For $a,b\in \mathbb R$ and $\epsilon_1 >\epsilon_2$  one considers  two sub-surjections $$ 
^1\tilde \pi  : A=^{BM} \mathbb F_r (B_1)\rightsquigarrow ^{BM}\mathbb F_r(B_2)= A'$$ and $$
^2\tilde \pi : A=( \mathbb G_{n-r} (\underline B_1))^\ast \rightsquigarrow \mathbb (G_{n-r}(\underline B_2))^\ast=A'$$ defined based on the boxes 
\vskip .1in

\begin{enumerate}[label= (\alph*)]
\item 
$B_1= (a-\epsilon_1, a+\epsilon _1]\times [b-\epsilon_1, b+\epsilon_1)$  

$B_2= (a-\epsilon_2, a+\epsilon _2]\times [b-\epsilon_2, b+\epsilon_2)$ 

$C= (a-\epsilon_2, a+\epsilon _1]\times [b-\epsilon_1, b+\epsilon_2)$ 

\noindent with $C$ as a lower-right corner of $B_1,$  $B_2$ as an upper-left corner of $C$ 

and on the boxes 

\item 
$\underline B_1= (b-\epsilon_1, b+\epsilon _1]\times [a-\epsilon_1, a+\epsilon_1)$  

$\underline B_2= (b-\epsilon_2, b+\epsilon _2]\times [a-\epsilon_2, a+\epsilon_2)$ 

$\underline C= (b-\epsilon_1, b+\epsilon _2]\times [a-\epsilon_2, a+\epsilon_1)$ 

\noindent with $\underline C$ as a upper left corner of $\underline B_1,$  $\underline B_2$ as an down right corner of $\underline C.$
\end{enumerate}

The boxes $\underline B_1, \underline C, \underline B_2$ are the symmetric of the boxes $B_1, C, B_2$ w.r. to the first diagonal $\Delta_1:=\{(x,y)\in \mathbb R^2 \mid x-y=0\}.$  See 
Figure 3 below.
\vskip .1in

\begin{tikzpicture} [scale=0.7]

\draw [line width=0.05cm] (-3,0) -- (10,0);
\draw [line width=0.05cm] (0,10) -- (0,-2);
\draw [dashed, line width=0.05cm] (0,11) -- (0,10);
\draw [dashed, line width=0.05cm] (0,-2) -- (0,-3);
\draw [dashed, line width=0.05cm] (-4,0) -- (-3,0);
\draw [dashed, line width=0.05cm]  (10,0) -- (11,0);
\draw [<-, line width=0.10cm] (-2,5) -- (6,5);
\draw [<-, line width=0.10cm] (6,9) -- (6,5);
\draw [dashed, line width=0.05cm] (-2,5) -- (-2,9);
\draw [dashed, line width=0.05cm] (-2,9) -- (6,9);
\node at (-1,8) {$B_1$};

\draw [<-, line width=0.10cm] (1,6) -- (4,6);
\draw [<-, line width=0.10cm] (4,8) -- (4,6);
\draw [dashed, line width=0.05cm] (1,5) -- (1,8);
\draw [dashed, line width=0.05cm] (1,8) -- (6,8);
\node at (-1,8) {$B_1$};
\node at (1.5,6.5) {$B_2$};
\node at (5.5,7) {$C$};
\node at (7,2.5) {$*  (b,a)$};

\draw [dashed, line width=0.05cm] (5,-2) -- (5,6);
\draw [dashed, line width=0.05cm] (9,6) -- (5,6);
\draw [<-, line width=0.10cm] (5,-2) -- (9,-2);
\draw [<-, line width=0.10cm] (9,6) -- (9,-2);
\node at (-1,8) {$B_1$};

\draw [dashed, line width=0.05cm] (6,1) -- (6,4);
\draw [dashed, line width=0.05cm] (8,4) -- (6,4);
\draw [dashed, line width=0.05cm] (5,1) -- (8,1);
\draw [dashed, line width=0.05cm] (8,1) -- (8,6);
\draw [dashed, line width=0.05cm] (5,1) -- (6,1);
\draw [<-, line width=0.10cm] (6,1) -- (8,1);
\draw [<-, line width=0.10cm] (8,4) -- (8,1);
\draw [line width=0.02cm] (-1.5,-1.5) -- (9,9);
\draw [dashed, line width=0.03cm] (-1.5,-1.5) -- (-2,-2);
\draw [dashed, line width=0.03cm] (9.5,9.5) -- (9,9);

\node at (8,-1) {$\underline B_1$};
\node at (7,1.5) {$\underline B_2$};
\node at (7,5.5) {$\underline C$};
\node at (2.5,7) {$*  (a,b)$};
\node at (2,-1) {Figure 3};
\node at (8.5,9) {$\Delta_1$};
\end{tikzpicture}
\vskip .1in 
\noindent In view of Observation \ref{O3.2} \  

 $\pi_{B_1}^C: ^{BM} \mathbb F_r(B_1)\to ^{BM}\mathbb F_r(C)$ is surjective  and 
 
 $i_{B_2} ^C: ^{BM} \mathbb F_r(B_2)\to ^{BM}\mathbb F_r(C)$ is injective and identifies  $^{BM}\mathbb F_r(B_2)$ to a subspace of $^{BM}\mathbb F_r(C).$ 
 
 Define the sub-surjection $^1\tilde \pi$ by:
$A:=^{BM} \mathbb F_r (B_1), P=^{BM}\mathbb F_r(C) , A' = ^{BM}\mathbb F_r(B_2)$ and $\pi := \pi_{B_1}^C.$  

 \noindent In view of Observation \ref{O3.2}  

$(i_{\underline  C}^{\underline B_1})^\ast : (\mathbb G_{n-r}(\underline B_1))^\ast \to (\mathbb  G_{n-r}(\underline C))^\ast$ is surjective 
and  

$(\pi_{\underline  C}^{\underline B_2})^\ast : (\mathbb G_{n-r}(\underline B_2))^\ast \to (\mathbb  G_{n-r}(\underline C))^\ast$ is injective 
\footnote { since $ i_{\underline  C}^{\underline B_1}: \mathbb G(\underline C) \to \mathbb  G(\underline B_1)$ is injective and 
$\pi_{\underline  C}^{\underline B_2} : \mathbb G(\underline C) \to \mathbb  G(\underline B_2)$ is surjective}.  

Define the second  sub-surjection $^2\tilde \pi$ by:
$A:=(\mathbb G_{n-r} (\underline B_1))^\ast , P=(\mathbb G_{n-r}(C))^\ast , A' = (\mathbb G_{n-r}(B_2))^\ast $ and $\pi := (i_{\underline  C}^{\underline B_1})^\ast .$  
 In view of (\ref{E36}) Poincar\'e duality identifies the two sub-surjections.  
\vskip .1in

For $a<b,  a,b\in \mathbb R$ and $\epsilon_1 >\epsilon_2$ with $a<2\epsilon _1 < b$ one considers  two sub-surjections 
$ 
^3 \tilde \pi   : A=^{BM} \mathbb T_r ^f (B_1)\rightsquigarrow ^{BMT} \mathbb T^f(B_2)$ 
and $ ^4 \tilde \pi : (\mathbb T^{-f}_{n-r-1}(\underline B_1))^\ast \rightsquigarrow (\mathbb T^{-f}_{n-r-1}(\underline B_2))^\ast$  based on the boxes above diagonal.
\vskip .1in

\begin{enumerate}[label= (\alph*)]
\item
$B_1= (a-\epsilon_1, a+\epsilon _1]\times (b-\epsilon_1, b+\epsilon_1]$  

$B_2= (a-\epsilon_2, a+\epsilon _2]\times (b-\epsilon_2, b+\epsilon_2]$ 

$C= (a-\epsilon_2, a+\epsilon _1]\times [b-\epsilon_2, b+\epsilon_1]$ 

\noindent with $C$ as a upper-right corner of $B_1$  and $B_2$ as the lower-left   corner of $C$ and on the boxes above diagonal

\item 
$\underline B_1= (-b-\epsilon_1, -b+\epsilon_1]\times (-a-\epsilon_1, -a+\epsilon _1]$  

$\underline B_2=  (-b-\epsilon_2, -b+\epsilon_2)\times (-a-\epsilon_2, -a+\epsilon _2]$ 

$\underline C= (-b-\epsilon_1, -b+\epsilon_2)\times (-a-\epsilon_1, -a+\epsilon _2]$ 

\noindent with $\underline C$ as a lower left corner of $\underline B_1,$  $\underline B_2$ as an upper-right corner of $\underline C$
\end{enumerate}
The boxes $\underline B_1, \underline C, \underline B_2$ are the symmetric  of the boxes $B_1 C, B_2$ w.r. to the second diagonal 
$\Delta_2:=\{(x,y)\in \mathbb R^2 \mid x+y=0\}.$ See 
Picture 4 below.
\vskip .2in
\begin{tikzpicture} [scale=0.6]

\draw [line width=0.05cm] (-11,0) -- (9,0);
\draw [dashed, line width=0.05cm] (-11,0) -- (-12,0);
\draw [dashed, line width=0.05cm]  (9,0) -- (10,0);
\draw [<-, line width=0.10cm] (-10,5) -- (-6,5);
\draw [<-, line width=0.10cm] (-6,-3) -- (-6,5);
\draw [<-, line width=0.10cm] (-9,3) -- (-7,3);
\draw [<-, line width=0.10cm] (-10,3) -- (-9,3);
\draw [<-, line width=0.10cm] (-7,-1) -- (-7,3);
\draw [<-, line width=0.10cm] (-7,-3) -- (-7,-1);
\draw [dashed, line width=0.05cm]  (-6,-3) -- (-10,-3);
\draw [dashed, line width=0.05cm]  (-10,-3) -- (-10,5);
\draw [dashed, line width=0.05cm]  (-9,-1) -- (-7,-1);
\draw [dashed, line width=0.05cm]   (-10,-3) -- (-6,-3);
\draw [dashed, line width=0.05cm]  (-9,3) -- (-9,-1);
\node at (-8,4.5) {$\underline B_1$};
\node at (-8.5,2) {$\underline B_2$};
\node at (-8,1) {$ * $};
\node at (-8.1, 0.5){$(-b,-a)$};
\node at (-8.5,9) {$\Delta_2$};

\draw [line width=0.02cm] (-9,9) -- (4,-4);
\draw [line width=0.02cm] (9,9) -- (-4,-4);
\draw [dashed, line width=0.05cm] (-5,10) -- (-5,6);
\draw [dashed, line width=0.05cm] (3,6) -- (-5,6);
\draw [dashed, line width=0.05cm] (-3,9) -- (-3,7);
\draw [dashed, line width=0.05cm] (-3,10) -- (-3,9);
\draw [dashed, line width=0.05cm] (1,7) -- (-3,7);
\draw [dashed, line width=0.05cm] (-3,7) -- (1,7);
\draw [dashed, line width=0.05cm]  (3,6) -- (3,10);
\draw [dashed, line width=0.05cm] (1,7) -- (3,7);
\draw [<-, line width=0.10cm]  (-5,10) -- (3,10);
\draw [<-, line width=0.10cm]  (1,7) -- (1,9);
\draw [<-, line width=0.10cm]   (3,6) -- (3,10);
\draw [<-, line width=0.10cm]  (-3,9) -- (1,9);
\node at (-8,4.5) {$\underline B_1$};
\node at (-8.5,2) {$\underline B_2$};
\node at (-1,8) {$*  (a,b)$};
\node at (-0.3,-3.5) {Figure 4};
\node at (-8.5,9) {$\Delta_2$};
\node at (8.5,9) {$\Delta_1$};
\node at (-4.5,6.5) {$B_1$};
\node at (-2.5,7.5) {$B_2$};
\node at (-7.5,-2.5) {$\underline C$};
\node at (2.5,8.5) {$C$};

\end{tikzpicture}
\vskip .2in

\noindent In view of Observation \ref{O3.4}   

$\pi_{B_1}^C: ^{BM} \mathbb T_{r}^f(B_1)\to ^{BM}\mathbb T_{r}^f(C)$ is  surjective  and 

$i_{B_2} ^C: ^{BM} \mathbb T_{r}^f(B_2)\to ^{BM}\mathbb T_{r}^f(C)$ is injective identifying $^{BM}\mathbb T_r^f(B_2)$ to  a subspace of $^{BM}\mathbb T_r^f(C).$ 

Define the first  sub-surjection $^3 \tilde \pi$ with $A:=^{BM} \mathbb T^f (B_1), P=^{BM}\mathbb T^f(C) , A' = ^{BM}\mathbb T^f(B_2)$ and $\pi := \pi_{B_1}^C.$  

\noindent In view of Observation \ref{O3.4} 

$(i_{\underline  C}^{\underline B_1})^\ast : (\mathbb T_{n-r-1}^{-f}(\underline B_1))^\ast \to (\mathbb  T_{n-r-1}^{-f}(\underline C))^\ast$ is surjective 
 \footnote {since $\pi_{\underline  C}^{\underline B_2} : \mathbb T^{-f}(\underline C)\to \mathbb  T^{-f}(\underline B_2)$ is surjective  and  $i_{\underline  C}^{\underline B_1} : \mathbb  T^{-f}(\underline C \to \mathbb T^{-f}(\underline B_1)) $ injective},

$(\pi^{\underline  C}_{\underline B_2})^\ast  : (\mathbb  T_{n-r-1}^{-f}(\underline B_2))^\ast \to (\mathbb T_{n-r-1}^{-f}(\underline C)^\ast $ is injective.

Define the  sub-surjection $^4\tilde \pi$ with $A:=(\mathbb T_{n-r-1}^{-f} (\underline B_1)^\ast,$   $P= (\mathbb T_{n-1-r}(\underline C))^\ast$ 
$A' = (\mathbb T_{n-r-1}^{-f}(\underline B_2)^\ast$ and $\pi= :(\pi^{\underline  C}_{\underline B_1})^\ast .$  

In view of (\ref{E39}) Poincar\'e duality identifies the two quasi surjections.  
\vskip .2in

Given $a,b\in CR(f)$  choose a sequence $\epsilon_1 > \epsilon_2 >\epsilon _3>\cdots \epsilon _k> \cdots >0$ with $\lim_{i\to \infty} \epsilon_i=0$ s.t. $a\pm \epsilon_i$ and $b\pm \epsilon_i$ are regular values for any $i.$ This is possible since the set of regular values of $f$ is everywhere dense in $\mathbb R.$ In case $a <b$ assume in addition that $a <2\epsilon_1 < b.$
  
First  consider the collection of sub-surjections $^{BMF} \tilde \pi_i : A_i\rightsquigarrow A_{i+1}$ with $A_i:= ^{BM} \mathbb F_r(B_i),$ $B_i$ the box $ B_i:= (a-\epsilon_i, a+\epsilon_i ]\times [b-\epsilon _i, b+\epsilon _i)$ which identifies  by Poincar\'e duality    to the collection of sub-surjections $^{G^\ast} \tilde \pi_i : A_i\rightsquigarrow A_{i+1}$ with $A_i:= (\mathbb G_{n-r}(\underline B_i))^\ast$  with $\underline B_i$  with the box $\underline B_i=  (b-\epsilon _i, b+\epsilon _i]\times [a-\epsilon_i, a+\epsilon_i ).$  
Since in view of (\ref{E36})
$\varinjlim_{i\to \infty}   {^1 \tilde \pi_i }    = \varinjlim_{i\to\infty}  {^2 \tilde \pi_i},$   Proposition \ref{P32} below implies $$^{BM} \hat \delta^f_r(a,b)= ( ^G \hat \delta ^f_{n-r}(b,a))^\ast.$$ 
In view of the finite dimensionality of $\hat \delta^f_{n-r}(b,a)= ^F \hat \delta^f_{n-r}(b,a)= ^G \hat \delta^f_{n-r}(b,a)$
 one has $\dim \   (^G \hat \delta^f_{n-r}(b,a))^\ast = \dim  \ ^G \hat \delta^f_{n-r}(b,a),$
hence $^{BM} \delta^f_r(a,b)= \delta^f_{n-r}(b,a)$ hence $$^{BM} \delta^\omega_r(t)= \delta^\omega_{n-r} (-t).$$ This establishes Item 1 in Theorem (\ref{TP}).   
\vskip .2 in

Similarly, given $a<b$  consider the collection of sub-surjections $^3 \tilde \pi^f_i : A_i\rightsquigarrow A_{i+1}$ with 
$A_i:= ^{BM} \mathbb T^f_r(B_i),$ $B_i$ the box above diagonal $ B_i:= (a-\epsilon_i, a+\epsilon_i ]\times (b-\epsilon _i, b+\epsilon _i]$ which identifies  by Poincar\'e duality   to the collection of sub-surjections $ ^4 \pi^{-f} _i : A_i\rightsquigarrow A_{i+1}$ with $A_i:= (\mathbb T^{-f}_{n-r-1}(\underline B_i))^\ast$  with $\underline B_i$   the box above diagonal $\underline B_i= 
(-b-\epsilon _i, -b+\epsilon _i]\times (-a-\epsilon_i, a+\epsilon_i ].$  

In view of (\ref{E39}) one has 
$\varinjlim_{i\to \infty} {^3 \tilde \pi^f_i }    = \varinjlim_{i\to\infty}  {^4 \tilde \pi^{-f}_i}.$  

Proposition \ref{P32} below implies $$^{BM} \hat \gamma ^f_r(a,b)= ( \hat \gamma ^{-f}_{n-r-1}(-b, -a))^\ast$$ 
and 
in view of the finite dimensionality of $\hat \gamma^{-f}_{n-r-1}(-b,-a)$ one has 

\noindent $\dim  \hat \gamma^{-f}_{n-r-1}(-b,-a)= (\dim  \hat \gamma^{-f}_{n-r-1}(-b,-a))^\ast$
hence $^{BM} \gamma^f_r(a,b)= \gamma^{-f}_{n-r-1}(-b,-a)$ hence $$^{BM} \gamma^\omega_r(t)= \gamma ^{-\omega}_{n-r-1} (t).$$ This establishes Item 2 in Theorem  \ref{TP}.

\begin {proposition} \label {P32}\
\begin{enumerate}
\item $\varinjlim_{i\to \infty}   {^1 \tilde \pi_i}= ^{BM} \hat \delta^f_r(a,b) $ 
\item $\varinjlim_{i\to \infty} {^2\tilde \pi_i }= (\hat \delta^f_{n-r} (b,a))^\ast$
\item $\varinjlim_{i\to \infty} {^3 \tilde \pi^f_i }= ^{BM} \hat \gamma ^f_{r}(a,b) $ 
\item $\varinjlim_{i\to \infty} {^4 \tilde \pi^{-f}_i}= ( \hat \gamma ^{-f}_{n-r-1}(-b, -a))^\ast $ 
\end{enumerate}
\end{proposition}
Proof: 

Using the description of $A^\infty_i$ (at the end of section 1)  Observations \ref{O3.2} and Observation \ref{O3.4}   one concludes  that for each directed system $^1 \tilde \pi_i,$
$^2 \tilde \pi_i,$  $^3 \tilde \pi_i,$$^4 \tilde \pi_i$ the corresponding $A^\infty _i$ are given by 
$$A^\infty _i=  \begin{cases}
^{BM} \mathbb F_r( (a-\epsilon_i, a]\times [ b, b+\epsilon_i)\\
( \mathbb G_{n-r} ((-b-\epsilon_i, b]\times [a. a+\epsilon_i)))^\ast\\
^{BM} \mathbb T^t_r((a-\epsilon_i, a]\times (b-\epsilon_i, b])\\
(\mathbb T_{n-r-1}^{-f}((-b-\epsilon_i, -b]\times (-a-\epsilon_i, -a]))^\ast .
\end{cases}
$$ 
 
In view of the Definitions \ref{D31} and of the definitions of $^{BM}\hat \delta_{\cdots}  ^f,  \  ^{G}\hat \delta_{\cdots} ^f, \  ^{BM}\gamma^f_{\cdots }$ one has

$$\varinjlim_{i\to \infty}   (\pi'_i:A^\infty_i \to A^{\infty}_{i+1})= \begin{cases} ^{BM} \hat\delta^f_r(a,b)\\
(^G \hat \delta^f_{n-r} (b,a))^\ast\\  ^{BM} \hat\gamma^f_r(a,b) \\ (\hat\gamma^{-f}_r(-b,-a))\ast
\end{cases} $$

q.e.d

\section {Proof of stability  property, Theorem \ref{TS} }

In view of Proposition {\ref {P3.5}} Note that for each diagonal $\Delta_t:=\{(x,y)\in \mathbb R^2, y-x=t\}$ the set $\Delta_t \cap \supp \delta^f_r$ if not empty is $\Gamma-$invariant w.r.  to the diagonal action and the set 
$(\Delta_t \cap \supp \delta^f_r) / \Gamma$ consists of finitely many orbits say $o_1, o_2, \cdots o_k.$ Choose in any such orbit one point $(a_i, b_i),$ and let 
 $ \delta^\omega_r(t)= \sum_{i=1,2,\cdots k} \delta^f_r(a_i, b_i).$
 Clearly $\delta^\omega_r(t)$ does not depend on the choice of $(a_i, b_i).$ Recall that $\hat\delta^f_r(a,b)\ne 0$ implies  that both $a$ and $b$ are critical values 

Denote by 

\begin{enumerate}

\item $\check{ \mathbb F}^f_r(\{a,b\}):= \oplus _{g\in\Gamma}  \hat \delta^f_r(a+g,b+g),$

\item $\check{ \mathbb F}^\omega_r(t):= \oplus _{\{(a,b)\in \mathbb R^2, b-a = t\}}  \hat \delta^f_r(a,b)= \oplus _{\{(a,b)\in \supp\ \delta^f_r, b-a = t\}}  \hat \delta^f_r(a,b)= \oplus _{i\mid (a_i, b_i)\in \Delta^\delta_t} \check {\mathbb F}^f_r(\{a_i,b_i\}),$

\item $\check{ \mathbb F}^\omega_r(\leq t):= \oplus _{\{(a,b)\in \mathbb R^2,  b-a\leq t\}}  \hat \delta^f_r(a,b)= \oplus _{\{(a,b)\in \supp\ \delta^f_r,  b-a\leq t\}}  \hat \delta^f_r(a,b)=\oplus_{s\leq t} \check {\mathbb F}^\omega_r(s),$

\item $\check{ \mathbb F}^\omega_r(<t):=\oplus _{\{(a,b)\in \mathbb R^2,  b-a < t\}}  \hat \delta^f_r(a,b)=\oplus _{\{(a,b)\in \supp\ \delta^f_r,  b-a < t\}}  \hat \delta^f_r(a,b)= \oplus_{s\leq t} \check {\mathbb F}^\omega_r(s),$

\item 
$\check{ \mathbb F}^\omega_r:=\oplus _{\{(a,b)\in \mathbb R^2\}}  \hat \delta^f_r(a,b)=  \oplus _{\{(a,b)\in \supp\ \delta^f_r\}}  \hat \delta^f_r(a,b)= \oplus_{s\in \mathbb R} \check {\mathbb F}^\omega_r(s).$

\end{enumerate}

\noindent Let $\langle g\rangle_{a,b}: \hat \delta^f_r(a,b)\to \hat \delta^f_r(a+g,b+g)$ be the linear isomorphism derived from the linear isomorphism $\langle g\rangle: H_r(\tilde X)\to H_r(\tilde X)$ 
end equip $\check{ \mathbb F}^\omega_r$ with a structure of $\kappa[\Gamma]-$ module  given by 
 $$\langle g \rangle:= \oplus _{(a,b)} \langle g\rangle_{a,b} \hat \delta^f_r(a,b)\to \hat \delta^f_r(a+g, b+g).$$
The $\kappa[\Gamma]-$modules 
$\check{ \mathbb F}^f_r(\{a,b\}),$\  $\check{ \mathbb F}^\omega_r(t),$ \ $\check{ \mathbb F}^\omega_r(\leq t),$\ $\check{ \mathbb F}^\omega_r(< t)$ and $\check{ \mathbb F}^\omega_r,$ \ 
are all free f.g. $\kappa [\Gamma]-$ modules and  in view of Proposition \ref{P3.5}  of finite rank 

\begin{enumerate}

\item $ \delta^f_r(a,b),$    

 \item $ \delta^\omega_r(t)= \sum_{i\mid , b_i-a_i \in \Delta^\delta_t } \delta^f_r(a_i, b_i),$  
  
\item  $\sum_{ s\in \supp \ \delta_r^\omega, s\leq t} \delta^\omega _r(s),$

 \item $\sum_{ s\in \supp \ \delta_r^\omega, s< t} \delta^\omega _r(s),$

 \item $\sum_{s\in \supp \ \delta_r^\omega} \delta^\omega _r (s).$\ 
\end{enumerate}

\noindent Note that $Tor H_r(\tilde X)= \cap_t \mathbb I^f_tr)= \cap_t \mathbb I_f^t(r)$ which implies that $\mathbb F_r(a,b) \supseteq  Tor H_r(H_r(\tilde X)) \subset H_r(\tilde X) / Tor (H_r(\tilde X))$ 

Denote by 

$\overline { \mathbb F}^\omega_r:= H_r(\tilde X)/ Tor H_r(\tilde X),$

$\overline { \mathbb F}^f_r(\{a,b\}):= \sum _{g\in\Gamma}   (F^f_r(a+g,b+g)/ Tor H_r(\tilde X))  \subseteq \overline {\mathbb F}^\omega_r,$

$\overline { \mathbb F}^\omega_r(t):= \sum _{\{(a,b)\in \supp\ \delta^f_r, b-a = t\}}  (F^f_r(a,b)/ Tor H_r(\tilde X)) \subseteq \overline {\mathbb F}^\omega_r,$ 

$\overline { \mathbb F}^\omega_r(\leq t):= \sum _{(a,b)\in \supp\ \delta^f_r,  b-a\leq t}  (F^f_r(a,b)/ Tor H_r(\tilde X)) \subseteq \overline {\mathbb F}^\omega_r,$

$\overline { \mathbb F}^\omega_r(<t):= \sum _{\{(a,b)\in \supp \ \delta^f_r,  b-a < t\}} (F^f_r(a,b)/ Tor H_r(\tilde X)) \subseteq \overline {\mathbb F}^\omega_r.$
\vskip .1in

\noindent Since $\kappa[\Gamma]$ is Noetherian and $H_r(\tilde X)$ is f.g.module all these $\kappa-$vector spaces are  f.g $\kappa[\Gamma]-$ modules. 
Consequently there are finitely many $t\in \mathbb R$, \ $t_1 <t_2 \cdots  <t_k,$ $t_i\in \supp \delta^\omega_r$ s.t 
\begin{enumerate}
\item $\overline{\mathbb F}^\omega_r(\leq t)= 0$ for $t< t_1,$
\item $\overline{\mathbb F}^\omega_r(\leq t)= \mathbb F^\omega_r(\leq t_i)$ for $t_i\leq t <t_{i+1},$
\item $\overline{\mathbb F}^\omega_r(\leq t)= H_r(\tilde X)/ Tor H_r(\tilde X)$ for $t_k\leq t .$
\item $\overline{\mathbb F}^\omega_r(\leq t_{i})/ \overline{\mathbb F}^\omega_r(\leq t_{i-1})=\overline{\mathbb F}^\omega_r( t_i).$
\end{enumerate}

Choose a collection of $\Gamma$ compatible splittings $S^\delta:=\{ i_{a,b} (r):  \hat \delta^f_r(a, b) \to \mathbb F_r(a,b)  \mid (a,b)\in \supp \ \delta^f_r \}$ ( like in subsection 2.4 ) and 
consider $$^S I(r): \check{\mathbb F}^\omega _r \to \overline {\mathbb F}^f_r= H_r(\tilde X)/ Tor H_r(\tilde X) .$$  The map $^S I(r)$ is $\kappa[\Gamma]-$linear, in view of Proposition 4 in \cite{BU1} is injective and sends $\check{\mathbb F}^{\cdots}_{\cdots}$ into $\overline {\mathbb F}^{\cdots}_{\cdots}$, and  restricts to the maps

$^S I_{\{a, b\}}(r): \check{\mathbb F}^f (\{a,b\}) \to \overline {\mathbb F}^f_r(\{a,b\}),$

$^S I(r)_t: \check{\mathbb F}^f _r(t) \to \overline {\mathbb F}^\omega_r(t),$

$^S I(r)_{\leq t}: \check{\mathbb F}^f _r(\leq t) \to \overline {\mathbb F}^\omega_r(\leq t),$

$^S I(r)_{< t}: \check{\mathbb F}^f _r(<t) \to \overline {\mathbb F}^\omega_r(<t)$

\noindent all injective $\kappa[\Gamma]-$linear maps.

\begin {proposition}\label {P41}
The above maps are  all bijective, hence  $\overline {\mathbb F}^\omega_r(t),\ \overline {\mathbb F}^\omega_r(\leq t),\ \overline {\mathbb F}^\omega_r(< t), \ \overline {\mathbb F}^\omega_r$ are all f.g. free modules of rank 
\begin{enumerate}
\item 
$\delta^\omega_r(t), $
\item 
$ \sum _{\{s\in \supp\ \delta^\omega_r\mid s\leq t\}} \  \delta^\omega_r(s),$
\item 
$\sum _{\{s\in \supp\ \delta^\omega_r\mid s< t\}}\ \delta^\omega_r(s).$\
\item 
$\sum _{\{s\in \supp\ \delta^\omega_r\}} \   \delta^\omega_r(s),$
\end{enumerate}
\end{proposition}

In view of the above it suffices to check the surjectivity of $^S I(r).$  

Indeed for $ x\in H_r(\tilde X)$  define: 
\begin{itemize}
\item $\alpha (x):= \inf \{ a \mid  x\in \mathbb I^f_a(r) \},$  $\alpha (x)\in [-\infty, \infty)$  

with $\alpha(x)=-\infty$ if $x\in \cap_{a\in \mathbb R}  \mathbb I_a^f(r),$
\item $\beta (x):= \sup \{ b \mid  x\in \mathbb I^b_f(r)\},$ $\beta  (x)\in (-\infty, \infty]$ 

with  $\beta (x)=\infty$ if $x\in \cap_{b\in \mathbb R}  \mathbb I^b_f(r),$
\item $t^\delta(x)= \beta(x)- \alpha(x),$ $t^\delta (x)\in (-\infty, \infty].$
\end{itemize}

In view of Proposition 3.1 in \cite{BU1} $x\in Tor H_k(\tilde X)$ iff $t(x)=\infty.$
Observe that if $x\in H_r(\tilde X)$ with $\alpha(x)= a, \beta(x)= b,$  $t= t(x)=b-a,$
hence $x\notin Tor H_r(\tilde X),$ then it can be  written (in the presence of a collection of splittings)  as $x=i_{a,b} (\hat x) + x'$ with $\hat x= \pi_{a,b} (x)\in \hat \delta_r(a,b)$ and $x'\in H_r(\tilde X$ with $t(x') < t.$ This implies that in view of finite generation of $H_r(\tilde X)$ there exists a finite sequence  $t_0 > t_1 >t_2  \cdots >t_k$ and $x_i$ with $\alpha(x_i)= a_i, \beta(x_i)= b_i, t(x_i)= t_i $ s.t.
 $x= \sum _{0\leq i \leq k} i_{a_i, b_i}(\hat x_i) ,$ $\hat x_i\ne 0;$  
 otherwise  $H_r(\tilde X)$ contains an infinite sequence of submodules $\cdots \overline {\mathbb F}_r (t_i)\supset \overline {\mathbb F}_r (t_{i+1} \supset \cdots $ which contradicts the f.g property of $H_r(\tilde X).$   
 
 q.e.d.  
\vskip .1in

To finalize the proof of Theorem \ref{TS} 
one  follows  the arguments in \cite{BH} or in \cite {B} section 5 . Observe that for any two TC 1-forms $\omega_1, \omega_2$ in the same cohomology class and  two lifts $f_1: \tilde X\to \mathbb R$ of $\omega_1$ and $f_2: \tilde X\to \mathbb R$ of $\omega_2$ with $ || f_1- f_2 ||_\infty<\epsilon$ and $a,b\in \mathbb R$  one has:

\begin{equation}\label {E40}
\begin{aligned} 
\mathbb I^{f_1}_{a-\epsilon} (r) \subseteq \mathbb I^{f_2}_{a}(r) \subseteq \mathbb I^{f_1}_{a+\epsilon}(r)\\
\mathbb I_{f_1}^{b-\epsilon}(r) \supseteq \mathbb I_{f_2}^{b}(r) \supseteq \mathbb I_{f_1}^{b+\epsilon}(r)
\end{aligned}
\end{equation}
and therefore 
 $$\mathbb F^{f_1}_r(a-\epsilon, b +\epsilon) \subseteq \mathbb F^{f_2}_r(a,b)\subseteq \mathbb F^{f_1}_r(a+\epsilon, b -\epsilon)$$
and therefore for any $t\in \supp \delta_r^\omega $ 

\begin{equation}\label {E41}
\begin{aligned}
\overline{\mathbb F}_r^{\omega_1}(t-2\epsilon) \subseteq  \overline{\mathbb F}_r^{\omega_2}(t) &\subseteq  \overline{\mathbb F}_r^{\omega_1}(t+2\epsilon)\\
\overline{\mathbb F}_r^{\omega_1}(<(t-2\epsilon)) \subseteq  \overline{\mathbb F}_r^{\omega_2}(<t) &\subseteq  \overline{\mathbb F}_r^{\omega_1}(<(t+2\epsilon)).
\end{aligned}
\end{equation}

Clearly,  if $\overline{\mathbb F}_r^{\omega_1}(t-\epsilon) = \overline{\mathbb F}_r^{\omega_1}(t+\epsilon)$ resp.  $\overline{\mathbb F}_r^{\omega_1}(<(t-\epsilon)) = \overline{\mathbb F}_r^{\omega_1}(<(t+\epsilon))$ then $\overline{ \mathbb F}_r^{\omega_2}(t)=  \overline{\mathbb F}_r^{\omega_1}(t-\epsilon) =
 \mathbb F_r^{\omega_1}(t)$ resp. $\overline{ \mathbb F}_r^{\omega_2}(<t)=  \overline{\mathbb F}_r^{\omega_1}(<(t-\epsilon)) =
 \mathbb F_r^{\omega_1}(<t).$  

For a tame TC1-form $\omega$ denote by $\sigma_r(\omega):=\inf |t_i- t_j|,  t_i\ne t_j$ with  $t_i, t_j\in \supp \delta^\omega_r.$ 
As an immediate consequence of  
Proposition (\ref {P41}) one has Proposition (\ref{P52}) which, as in \cite{BH} or \cite {B} section 5, 
implies the continuity of the assignment $\omega \rightsquigarrow \delta^\omega_r.$ 
\vskip .1in

\begin{proposition} \label {P52}

For any $\epsilon <\sigma_f(\omega)$, $\omega, \omega'$  two tame TC1-forms in the same cohomology class with 
$|\omega-\omega'| <\epsilon/3$  and $t_i\in \supp {\hat \delta}^\omega_r$ one has  
\begin{equation*} \delta^\omega_r(t_i)= \sum  _{t_i-\epsilon< s <t_i+\epsilon}  \delta ^{\omega'}_r(s)
 \end{equation*}
\begin{equation*}
\supp\ \delta^{\omega'}_r  \subset \cup _{t_i\in \supp\ \delta_r^\omega} (t_i-\epsilon, t_i+\epsilon).
\end{equation*}
In fact 
$${\hat \delta}^\omega_r(t_i)\simeq \bigoplus_{t_i-\epsilon< s <t_i+\epsilon}  {\hat \delta} ^{\omega'}_r(s).$$
 \end{proposition}

\section {Standard homology versus Borel-Moore homology (Proposition \ref{PP2})} 

 In this section $H_r$ denotes the standard (singular) homology, $^{BM} H_r$ the Borel-Moore homology  and the linear maps  $\theta^{X,Y}_r:  H_r(X,Y) \to  ^{BM} H_r(X,Y)$ 
define the natural transformation from the standard to the Borel-Moore homology.
In general the linear maps $\theta^{X,Y}_r:  H_r(X,Y) \to  ^{BM} H_r(X,Y)$  are  not isomorphisms. 

Consider pairs $(U, V)$ of locally compact ANRs, $V$ closed subset of $U.$
An inclusion of such pairs pairs 
$(U_1, V_1)\subset  (U_2, V_2),$  
 s.t. $U_1$  open subset set in $U_2$ and $V_1$ open set of   $V_2$
 induces the linear maps  
$i_{U_1, V_1} ^{U_2, V_2} : H_r (U_1, V_1) \to H_r(U_2, V_2)$ and $p_{U_2, V_2}^{U_1,V_1} : ^{BM} H_r(U_2, V_2) \to ^{BM} H_r(U_1, V_1)$  which together with $\theta_r^{U_1,V_1}$ and $\theta^{U_2, V_2}_r$ make  the diagram  
\begin{equation} \label {D42}
\xymatrix{H_r (U_1,V_1)\ar[r]^{\theta_r^{U_1, V_1}}\ar[d]_ {i^{U_2,V_2}_{U_1,V_1}}
&^{BM} H_r(U_1,V_1) \\ H_r (U_2, V_2)\ar[r]^{\theta_r^{U_2, V_2}}&^{BM} H_r(U_2, V_2)\ar[u] _{p_{U_2,V_2} ^{U_1, V_1}}
}
\end{equation} commutative.  When $V_1$ and $V_2$ are the empty sets one has the commutative diagram 
\begin{equation*}
\xymatrix{H_r (U_1)\ar[r]^{\theta_r^{U_1}}\ar[d]_ {i^{U_2}_{U_1}}
&^{BM} H_r(U_1) \\ H_r (U_2)\ar[r]^{\theta_r^{U_2}}&^{BM} H_r(U_2)\ar[u] _{p_{U_2} ^{U_1} }.
}
\end{equation*} 

By similar arguments as in \cite{Mi}, when regarded as a contravariant functor on the category whose objects are  open subsets $U$  of a locally compact ANR $X$  and morphisms are the inclusions, the assignment  
$U\rightsquigarrow ^{BM} H_r(U)$ induces the exact sequence 
\begin{equation} \label {E43}
\xymatrix{0\ar[r] &{\varprojlim '}_{i\to \infty}  \ ^{BM} \mathcal H_{r-1}(U(i)) \ar[r] 
&^{BM} \mathcal H_{r}(X)\ar[r]  &\varprojlim_{i\to \infty} \  ^{BM} \mathcal H_r(U(i))\ar[r]&0}\end{equation} 
 for any open filtration of $X$  by open sets $U(i),$ 
$$ U(1) \subset U(2)\subset  \cdots \subset U(i)\subset U(i+1)\subset \cdots \subset X,$$ $X= \cup_i U(i).$ Here $\varprojlim$  and 
$\varprojlim' $ denotes  {\it inverse limit} and the {\it derived inverse limit} of the  system of vector spaces and linear maps    
$$ p_{U_{i+1}}^{U_i}: ^{BM} H_r(U_{i+1})\to ^{BM} H_r(U_i).$$
For a  pair $(X,Y)$ as above equipped with a filtration by open pairs,  
$$ (U(1), V(1)) \subset (U(2),V(2))\subset  \cdots \subset (U(i), V(i)) \subset (U(i+1), V(i+1))\subset (X,Y),$$
i.e. $U(i)$ open subset of $X$, $V(i)$ open subset of $Y,$ $X= \cup_i (U(i), Y= \cup_i V(i),$ 
 the exact sequence  (\ref{E43}) implies the short exact sequence  

\begin{equation} \label {E44}
\xymatrix{0\ar[r] &{\varprojlim'}_{i\to \infty} \  ^{BM} \mathcal H_{r-1}(U(i), V(i)) \ar[r] 
&^{BM} \mathcal H_{r}( X, Y)\ar[r]  &\varprojlim_{i\to \infty} \  ^{BM} \mathcal H_r(U(i), V(i))\ar[r]&0}
\end{equation} 
with the {\it derived inverse limit} and  the {\it  inverse limit} considered for of the  system of vector spaces and linear maps  
$$p_{U_{i+1}, V_{i+1}} ^{U_i. V_i}: ^{BM} H_r(U(i+1), V(i+1))\to ^{BM} H_r(U(i), V(i)).$$

Consider $f:\tilde X \to \mathbb R$ a lift of a tame TC1-form $\omega$  and $$ U(1) \subset U(2)\subset  \cdots \subset U(i)\subset U(i+1)\subset \cdots \subset U(\infty)=\tilde X$$ a filtration by open sets of $\tilde X.$ 
By applying  the exact sequence (\ref{E44}) to $X= \tilde X_a$, $Y=\tilde X_{a-\epsilon}$  and the filtration by the pairs $(U(i)_a, U(i)_{a-\epsilon})$ of the pair $(\tilde X_a, \tilde X_{a-\epsilon})$ and by passing  to direct limit when $\epsilon \to 0$  one obtains the exact sequence
 
{\scriptsize
\begin{equation}\label {E45}
\xymatrix{0\ar[r] &{\varprojlim'}_{i\to \infty} \ ^{BM}\mathcal H_{r-1}(U(i)_a, U(i)_{<a}) \ar[r] 
&^{BM} \mathcal H_{r}( \tilde X_a, \tilde X_{<a})\ar[r]  &\varprojlim_{i\to \infty}\  ^{BM} \mathcal H_r(U(i)_a, U(i)_{<a})\ar[r]&0}
\end{equation} }
\vskip .1in

For the proof of Proposition \ref {PP2} we will also need the following Proposition (\ref {P6.1}). 
\begin{proposition} \label {P6.1}
Suppose $M$ is a compact manifold with boundary $\partial M$  and interior $Int M= U,$ and $f: M\to \mathbb R$  a tame map with 
with  $Cr(f)\subset U.$   
Suppose $t$ is a regular value for  $f_{\partial M}$  the restriction of $f$ to the boundary, but  possibly critical value for $f.$ Then  for $\epsilon>0$ small enough  the inclusion induced linear map $i_r(t):=i^{M_t, M_{t-\epsilon}}_{U_t, U_{t-\epsilon}}$ and the  $\theta_r(t):=\theta^{U_t, U_{t-\epsilon}}_r$ in the sequence (\ref{E26}) below 
are isomorphisms.
\begin{equation}\label {E46}
\xymatrix {H_r(M_t, M_{t-\epsilon})& H_r(U_t, U_{t-\epsilon})\ar[l]_{i_r(t)} \ar[r]^{\theta_r(t)} & ^{BM} H_r(U_t, U_{t-\epsilon})} 
\end{equation} 
 \end{proposition}

Proof:   We leave to the reader to check  that  the inclusion $(U_t, U_{t-\epsilon}) \subset (M_t, M_{t-\epsilon})$ is a homotopy equivalence of pairs  which implies $i_r(t)$ is an isomorphism. Concerning $\theta_r(t)$ observe  that
since $t$ is regular value for $f_{\partial M}$  there exists $\epsilon >0$ small enough s.t. the inclusion $\partial M(t-\epsilon)\subset \partial M_{[t-\epsilon, t]} \footnote {$\partial M_{[t-\epsilon, t]}= f_{\partial M} ^{-1} ([t-\epsilon, t])$}$  is a homotopy equivalence,  hence the inclusion $(M_t, M_{t-\epsilon})\subset (M_t, M_{t-\epsilon}\cup \partial M_{[t-\epsilon, t]})$ is a homotopy equivalence of compact pairs. Hence 

\begin{equation} \label{E47}
H_r(M_t, M_{t-\epsilon})=  H_r(M_t, M_{t-\epsilon}\cup \partial M_{[t, t-\epsilon]}) =^{BM} H_r(M_t, M_{t-\epsilon}\cup \partial M_{[t, t-\epsilon]}).
\end{equation}
Since  
$M_t \setminus  (M_{t-\epsilon}\cup \partial M_{[t-\epsilon, t]})= (U_t\setminus U_{t-\epsilon})$ one has 
\begin{equation} \label {E48}
 ^{BM} H_r(M_t, M_{t-\epsilon}\cup \partial M_{[t-\epsilon, t]})= ^{BM} H_r(U_t, U_{t-\epsilon}). 
 \end{equation}
 Combining the two isomorphisms one concludes that $\theta_r^{U_t, U_{t-\epsilon}}(t)\cdot i_r(t) ^{-1}$ is an isomorphism, hence so is $ \theta_r^{U_t, U_{t-\epsilon}}(t).$ 

q.e.d.

\vskip .1in 
{\bf Proof of Proposition (\ref{PP2})}
\vskip .2in 
For a fix $t\in \mathbb R$  choose a  filtration of $\tilde X$ by open sets 
$$U(1) \subset U(2) \subset\cdots \subset U(k) \subset U(k+1)\subset \cdots \subset \tilde X,$$ $\bigcup_i  U(i)= \tilde X,$   \ 
with the  properties:  
\begin{enumerate} [label= (\alph*)]
\item $\overline U(i)$ \footnote {$\overline U$ denotes the closure of $U$ in $\tilde X$ } compact submanifold with boundary  s.t. $\overline U(i) \subset U(i+1),$ 
\item the restrictions of $f$ to $\overline U(i)$ 
are  tame maps, 
\item  $t\in CR(f)$ is a regular value for the restriction of $f$ to $\partial (\overline U(i))$ 
\end{enumerate}  

In view of tameness of $f$ 

\noindent \hskip .1in (d): for any  $t\in CR(f)$ there exists $N(t)\in \mathbb Z_{\geq 0}$  s.t. 
the compact set $f^{-1}(t)\cap Cr(f) $
is contained in $U(i)$  for $i> N(t).$ 

Clearly such filtrations exists.  One can  choose  $\overline U(i)$ to be the closure of a finite union of translates of a  properly chosen fundamental domain  of the free action of $\Gamma$ on $\tilde X$  (i.e. whose interior is a manifold). It satisfies (a) and (b). One can increase (arbitrary little $\overline U(i)$ inside $U(i+1)$) to make (c) also satisfied. 
\vskip .1in

Observe that  
diagram  (\ref {D42})  implies the commutative diagram \footnote {recall 
$\varinjlim_{\epsilon\to 0} \  ^{BM}  H_r( U(i)_t, U(i)_{t-\epsilon})= ^{BM} \mathcal H_r( U(i)_t, U(i)_{<t})$}

\begin{equation} \label {D49}
\xymatrix{H_r (U(i)_t,U(i)_{<t})\ar[r]^{\theta_r^i(t)}\ar[d]_ {i^{i+k}_i} \ar@/_6pc/[dd]_{i^\infty_i}
&^{BM} H_r(U(i)_t,U(i)_{<t})\\ H_r (U(i+k)_t,U(i+k)_{<t})\ar[r]^{\theta_r^{i+k}(t)}\ar[d]_ {i^{\infty}_{i+k}}&^{BM} H_r(U(i+k)_t,U(i+k)_{<t})\ar[u] _{p_{i+k} ^i}
\\
H_r (\tilde X_t,\tilde X_{<t})\ar[r]^{\theta_r^{\infty}(t)}&^{BM} H_r(\tilde X_t,\tilde X_{<t})\ar[u] _{p_{\infty} ^{i+k}}\ar@/_6pc/[uu] _{p_\infty ^i}}
\end{equation}
s.t.  $i_j^{j+k+r}= i_{j+k}^{j+k+r} \cdot i_j^{j+k}$ and $p^j_{i+j+k}= p_{j+k}^j\cdot p_{j+k+r}^{j+k}.$
Inspection of this diagram combined with  Propositions \ref{P6.1}  and  the exact sequence (\ref{E45})  lead to the result as follows.  
\vskip .1in

F1. In view of Property (d) of the filtration, Proposition (\ref{P6.1})  implies that for $i >N(t)$ and  
$\epsilon$ small enough  $\theta^{U(i)_t, U(i)_{t-\epsilon}}_r : H_r (U(i)_t, U(i)_{t-\epsilon})\to  ^{BM} H_r (U(i)_t, U(i)_{t-\epsilon})$
is an isomorphism. Passing to direct limit when $\epsilon \to 0$  one obtains  
 $$\theta^i_r (t): H_r (U(i)_t, U(i)_{<t})\to  ^{BM} \mathcal H_r (U(i)_t, U(i)_{<t})$$ is an isomorphism.

F2. For $ j> N(t)$  the linear map $i_j^{j+k}$  is an isomorphism by excision property of the standard homology.   Combined with the commutativity of the diagram (\ref{D49}) and  F1. one obtains $p_{j+k} ^j$ is an isomorphism for any $j >N(t)$ and $k$ .

F.3  The isomorphisms $i^{j+k}_j$ stated in F2. imply $i^\infty _j$ is isomorphism for $j>N(t),$ and the isomorphisms $p^j_{j+k}$ for$j>N(t)$ stated in F2., in view of  the exact sequence ({E45}) 
imply  $p_\infty ^j$ is isomorphism for $j>N(t).$   Hence in view of diagram (\ref{D49}),  $\theta_r(t)= \theta^\infty_r(t)$ is an isomorphism. This finalizes the proof of Proposition (\ref{PP2}).

\end{document}